\documentclass[12pt]{article}

\usepackage{graphicx,pstricks}
\usepackage{amsmath,amsfonts,amssymb,amsthm}
\usepackage{mathrsfs}
\usepackage{textcomp,url}

\newtheorem{theorem}{Theorem}[section]
\newtheorem{lemma}[theorem]{Lemma}
\newtheorem{proposition}[theorem]{Proposition}
\newtheorem{corollary}[theorem]{Corollary}
\newtheorem{remark}[theorem]{Remark}

\newcommand{\E}{\mathbb{E}^{n+2}}
\newcommand{\R}{\mathbb{R}}

\newcommand{\Q}{\mathbb{Q}}
\newcommand{\QR}{\mathbb{Q}^{n}_{c}\times\mathbb{R}}
\newcommand{\dt}{\partial/\partial t}
\newcommand{\f}{f\colon M^{n}\rightarrow \mathbb{Q}^{n}_{c}\times\mathbb{R}}

\newcommand{\co}{C_{c}(\|T\|t)}
\newcommand{\s}{S_{c}(\|T\|t)}

\begin{document}

\title
{Hypersurfaces with constant principal curvatures in $\mathbb{S}^{n}\times\mathbb{R}$ and $\mathbb{H}^{n}\times\mathbb{R}$}

\author{ROSA CHAVES,\\
 Instituto de Matem\'atica e Estat\'istica,\\
  Universidade de S\~ao Paulo, S\~ao Paulo, SP, Brazil. 05508-090\\
  e-mail:~\url{rosab@ime.usp.br} \vspace*{4mm}\\
  ELIANE SANTOS,\\
  Instituto de Matem\'atica, \\
  Universidade Federal da Bahia, Salvador, BA, Brazil. 40170-110\\
   e-mail:~\url{elianesilva@ufba.br}
  }

\maketitle

\begin{abstract}
In this paper, we classify the hypersurfaces in $\mathbb{S}^{n}\times \mathbb{R}$ and $\mathbb{H}^{n}\times\mathbb{R}$, $n\neq 3$, with $g$ distinct constant principal curvatures, $g\in\{1,2,3\}$, where $\mathbb{S}^{n}$ and $\mathbb{H}^{n}$ denote the sphere and hyperbolic space of dimension $n$, respectively. We prove that such hypersurfaces are isoparametric in those spaces.
Furthermore, we find a necessary and sufficient condition for an isoparametric hypersurface in $\mathbb{S}^{n}\times \mathbb{R}\subset \mathbb{R}^{n+2}$ and $\mathbb{H}^{n}\times\mathbb{R}\subset \mathbb{L}^{n+2}$ with flat normal bundle, having constant principal curvatures.
\end{abstract}



\section{Introduction}

The study of hypersurfaces in product spaces has attracted the attention of many geometers in recent years.
First the surfaces with  constant mean curvature and more particularly the minimal surfaces in product spaces were studied in works of H. Rosenberg, W. Meeks and U. Abresch, \cite{rosenberg2}, \cite{meeks} and  \cite{rosenberg1}.
They were also studied by  I. Onnis and S. Montaldo in  \cite{irene2}, \cite{irene1}, \cite{irene3} and  B. Nelli in  \cite{nelli1}, between many others.

In  \cite{espinar1} and  \cite{espinar2}, J. Aledo, J. Espinar and J. G\'alvez described the surfaces with constant Gaussian curvature  in $\mathbb{S}^{2}\times\mathbb{R}$ and  $\mathbb{H}^{2}\times\mathbb{R}$. Moreover, J. Espinar,  J. G\'alvez and   H. Rosenberg,  showed in  \cite{espinar3} that a complete surface with constant  positive extrinsec curvature  in   $\mathbb{S}^{2}\times\mathbb{R}$  and  $\mathbb{H}^{2}\times\mathbb{R}$ is a rotational sphere.

 In order to unify the notations we are going to denote by  $\mathbb{Q}^{n}_{c}$ the sphere  $\mathbb{S}^{n}$, if  $c=1$ and the hyperbolic space  $\mathbb{H}^{n}$, if  $c=-1$.

 The rotational hypersurfaces in  $\QR$ were parametrized by  F. Dillen, J. Fastenakels  and  J. Van Der Veken in \cite{rot} where they extended the work  of  M.P. do Carmo and  M. Dajczer   \cite{rotM} about rotational hypersurfaces in space forms.

The hypersurfaces in  $\QR$ having a special field  $T$ as a principal direction were locally classified by R. Tojeiro in  \cite{RT1}.
 The differentiable field  $T$ and the differentiable function  $\nu$ are defined by the equation   $$\dt=df(T)+\nu\eta,$$
 where  $f$ is an immersion of a Riemannian $n$-dimensional manifold $M^{n}$ in  $\QR$ with unit normal vector field  $\eta$ and  $\dt$ is an unitary vector field  tangent to  $\mathbb{R}$. More particularly, the hypersurfaces with   constant angle, i.e., the hypersurfaces with constant function $\nu$, were also classified  in  \cite{RT1}.

  As we can observe in \cite[Proposition 4]{RT1},  the hypersurfaces in  $\QR$ has flat normal bundle as an isometric immersion into $\E$ if and only if $T$ is a principal direction of $f$.

In \cite{Manfio}, R. Tojeiro and F. Manfio classified locally the hypersurfaces in $\mathbb{Q}^{n}_{c} \times \mathbb{R}$, $n~\geq~3$, with constant sectional curvature.

Motivated by these results, in this paper  we  investigate hypersurfaces $\f$ with constant principal curvatures.

 It is well known that a hypersurface in a space form is isoparametric  if and only if its principal curvatures are constant but this does not happen in other ambients, in general.  For instance in \cite{Wang} one can find examples  of isoparametric hypersurfaces in complex projective spaces  that do not  have constant principal curvatures. See also G. Thorbergsson \cite{Survey}.

 In order to analyze if, for hypersurfaces in $\QR$, is true the equivalence between to be isoparametric and to have constant principal curvatures, we obtain a necessary and sufficient condition presented in Theorem \ref{di} for such equivalence to occur in hypersurfaces that have $ T $ as a principal direction. For this, we prove in Theorem \ref{erdp}, the existence of a local frame of differentiable principal directions, a result that has been used previously in the literature. We consider a hypersurface that has the field  $ T $ as a principal direction, construct its family of parallel hypersurfaces  and relate their respective principal curvatures.

The main  purpose of this work is the classification given by Theorem \ref{teoclass}, of hypersurfaces of $\QR$, $n\neq3$, with $g$ distinct constant principal curvatures, $g\in\{1,2,3\}$. Initially, we obtain in Theorem \ref{ctdp}, the classification of hypersurfaces in $\QR$ that have constant principal curvatures  contained in the class  having $ T $ as a principal direction. We prove some results related to the multiplicities of such curvatures such as  Theorem \ref{m1} and Proposition \ref{dp} and finally we  obtain Theorem \ref{teoclass}.

\section{Preliminaries}\label{sec:preliminaries}

Let $\mathbb{Q}^{n}_{c}$ denotes either the sphere $\mathbb{S}^{n}$ or hyperbolic space $\mathbb{H}^{n}$, according as $c=1$ or $c=-1$, respectively.
We consider $$\Q^{n}_{c}=\{(x_{1},\ldots,x_{n+1}) \in \mathbb{E}^{n+1} /cx_{1}^{2}+x_{2}^{2}+\cdots+x_{n+1}^{2}=c\},$$ with $x_{1}>0$ if $c=-1$ and  $$\mathbb{E}^{n+1}=\{(x_{1},\ldots,x_{n+2}) \in \E /x_{n+2}=0\},$$ where
 we  denote by $\E$ either the  Euclidean space $\R^{n+2}$ or the Lorentzian space $\mathbb{L}^{n+2}$ of dimension (n+2), according as  $c=1$ or $c=-1$, respectively.  Here  $(x_{1}, . . . , x_{n+2})$ are the standard coordinates on $\E$  and the flat metric $\langle \,,\,\rangle$  in those coordinates is written as $ds^{2}= c\,dx_{1}^{2}+ \ldots+ dx_{n+2}^{2}$.

 Given a hypersurface $\f $, let $\eta$ denote a unit vector field  normal to $f$ and let $\dt$ denote  a unit vector field tangent to the second factor $\R$. We define the differentiable vector field $T\in TM^{n}$ and a smooth function $\nu$ on $M^{n}$ by
\begin{equation}\label{eqtv}
\dt= df(T)+ \nu \eta.
\end{equation}
Since $\dt$ is a unit vector field, we have
\begin{equation}\label{vt1}
\nu^{2}+\|T\|^{2}=1.
\end{equation}

Let $\nabla$  be the Levi-Civita connection, $R$ be the curvature tensor of $M^{n}$ and let $A$ be the shape operator of $f$ with respect to $\eta$. The fact that $\dt$ is parallel in $\QR$ yields for all $X\in TM^{n}$ that
\begin{eqnarray}
\label{eq03}
\nabla _{X}T=\nu AX \;\;\;\;\textrm{ and }\\
\label{eq04}
X(\nu)=-\langle AX,T\rangle.
\end{eqnarray}
Moreover, the Gauss and Codazzi equations are
\begin{equation}\label{eq01}
R(X,Y)Z= (AX\wedge AY)Z+ c ((X\wedge Y)Z-\langle Y,T\rangle(X \wedge T)Z+\langle X,T\rangle(Y\wedge T)Z),
\end{equation}
and
\begin{equation}\label{eq02}
(\nabla_{X}A)Y-(\nabla_{Y}A)X=c\nu (X\wedge Y)T,
\end{equation}
respectively, where $ (X\wedge Y)Z=\langle Y,Z\rangle X-\langle X,Z\rangle Y$ and $X, Y, Z \in TM^{n}$.
\begin{remark}\label{obsvctd}
By (\ref{eq04}), if $\nu$ is constant and $T\neq0$, then $T$ is a principal direction and the  principal curvature associated to it  is equal to $0$.
\end{remark}

 Consider  $F:M^{n}\rightarrow \E$ given by $F:=i \circ f$ where $i:\QR \rightarrow \E$ is the inclusion map  whose unit normal  field $\xi$  satisfies $\langle \xi,\xi\rangle=c$. If $A_{\xi}$ is the Weingarten operator of the immersion $F$ with respect to the normal direction $\xi$, we obtain $A_{\xi}(T)=-\nu^{2}T$ and $A_{\xi}(X)=-X,$ for all $X \in [T]^{\bot}$ where $[T]^{\bot}=\{X \in TM^{n}/\langle X, T\rangle=0\}$. Let $\widehat{\nabla}$ denote the Riemannian connection of $\mathbb{E}^{n+2}$.

\begin{proposition} The following equalities hold for all $X\in TM^{n}$,
 \begin{equation}\label{eqobs1}
\widehat{\nabla}_{X}\xi=df(X)-\langle X,T\rangle\dt,
\end{equation}
\begin{equation}\label{eqobs2}
 \nabla^{\perp}_{X}\xi=-\nu\langle X,T\rangle\eta,
 \end{equation}
 \begin{equation}\label{eqobs3}
  \nabla^{\perp}_{X}\eta=c\nu\langle X,T\rangle\xi.
  \end{equation}
\end{proposition}

Two trivial classes of hypersurfaces of $\QR$ arise if either $T$ or $\nu$ vanishes identically.  Both classes will appear in our results.

\begin{proposition}\label{triviais}\cite[Proposition 1]{Manfio}
 Let $\f$ be a hypersurface.\\
(i) If $T$ vanishes identically, then $f(M^{n})$ is an open subset of a slice $\mathbb{Q}^{n}_{c}\times \{t\}$.\\
(ii) If $\nu$ vanishes identically, then $f(M^{n})$ is an open subset of a Riemannian product $M^{n-1}\times \R$, where $M^{n-1}$ is a hypersurface of $\Q^{n}_{c}$.
\end{proposition}

 In \cite{totgeod}  one can find a theorem that classifies the totally geodesic hypersurface of $\mathbb{S}^{n}\times \mathbb{R}$. The same result also holds for $\mathbb{H}^{n}\times \mathbb{R}$.

\begin{theorem}\label{totgd}\cite[Theorem 3] {totgeod} Let $M^{n}$ be a totally geodesic hypersurface of $\mathbb{S}^{n}\times \mathbb{R}$. Then $M^{n}$
is an open part of a hypersurface $S^{n}\times\{t_{0}\}$ for $t_{0}\in \mathbb{R}$, or of a hypersurface $S^{n-1}\times\mathbb{R}$.
 \end{theorem}

\section{Existence of a frame of principal directions}

 In order to prove the results of the next sections, we will need the following theorem  based on results that can be found in \cite[Theorem 2.6]{Md} and  \cite{Nomizu}. It is very important since it shows the existence of a local frame of differentiable principal directions.

\begin{theorem}\label{erdp} Let  $A$ be a symmetric tensor of type  $(1,1)$ in an oriented  Riemaniann manifold  $M^{n}$, $n\geq 2$, with  $g$ distinct eigenvalues  $\lambda_{1},\ldots,\lambda_{g}$ having constant multiplicities $m_{1},\ldots,m_{g}$, respectively. Then for each point  $p \in M$ there exist an orthonormal frame of differentiable eigenvalues $\{X_{1},\ldots, X_{n}\}$ defined in a neighborhood $U$ of  $p$ in $M$.
\end{theorem}

\begin{proof} Without loss of generality we may suppose $\lambda_{1}>\lambda_{2}>\ldots>\lambda_{g}$. Let us consider the $g$ orthogonal distributions $D_{\lambda_{i}}$ with $i \in \{1,\ldots,g\}$ defined by  $$D_{\lambda_{i}}(p)=\{Y_{p}\in T_{p}M; AY_{p}=\lambda_{i} Y_{p}\}.$$
 Given $p \in M$ let $U'$ be a neighborhood  of $p$ in  $M$ where are defined the  differentiable fields  $Y_{1},\ldots,Y_{n}$ such that
 $$\{Y_{1}^{1}(p),\ldots,\; Y_{m_{1}}^{1}(p)\}\;\;  \textrm{span the distribution}\;\; D_{\lambda_{1}}(p),$$
 $$\{Y_{m_{1}+1}^{2}(p),\ldots,Y_{m_{1}+m_{2}}^{2}(p)\}\;\; \textrm{span the distribution} \;\;  D_{\lambda_{2}}(p), \cdots,$$
 $$\{Y_{m_{1}+\cdots+m_{g-1}+1}^{g}(p),\ldots,Y_{n}^{g}(p)\}\;\; \textrm{span the distribution} \;\; D_{\lambda_{g}}(p).$$

Let us define  $X_{1},\ldots,X_{n}$ in  $U'$ by  $$X_{i}^{1}(x)=(A(x)-\lambda_{2} I)\ldots(A(x)-\lambda_{g} I)Y_{i}^{1}(x), \; \textrm{for}\; i\in\{1,\ldots,m_{1}\},$$ $$X_{i}^{2}(x)=-(A(x)-\lambda_{1} I)(A(x)-\lambda_{3} I)\ldots(A(x)-\lambda_{g} I)Y_{i}^{2}(x), \; \textrm{for}\; i\in\{m_{1}+1,\ldots,m_{1}+m_{2}\}$$ and $$X_{i}^{k}(x)=(-1)^{k-1}\prod_{j\neq k}(A(x)-\lambda_{j} I)Y_{i}^{k}(x),$$ for each  $k\in \{3,\ldots,g\}$ and  $i\in\{m_{1}+\cdots+m_{k-1}+1,\ldots,m_{1}+\cdots+m_{k}\}$, with $j\in\{1,\ldots,g\}$, where $I$ is the identity matrix  of ordem $n$ and  $x\in U'$.

 Observe that those fields depend on $x$ in a differentiable form. This happens because  the eingenvalues have constant multiplicity and so they are differentiable \cite{Nomizu}. Moreover, as  $\{X_{1}(p),\ldots,X_{n}(p)\}$ are linearly independent  then  $\{X_{1}(x),\ldots,X_{n}(x)\}$ are linearly independent  for all  $x$ in a neighborhood $U\subset U'$ of  $p$ in  $M$. Observe also that for each  $x\in U$ the basis  $\{X_{1}(x),\ldots,X_{n}(x)\}$ is  positive since it has the same orientation as the basis $\{Y_{1}(x),\ldots,Y_{n}(x)\}$.

The characteristic polinomium  of the operator  $A(x)$ is given by   $$p(t,x)=(t-\lambda_{1})^{m_{1}}(t-\lambda_{2})^{m_{2}}\ldots(t-\lambda_{g})^{m_{g}}$$ and by Cayley-Hamilton theorem it follows that  $$(A(x)-\lambda_{1} I)(A(x)-\lambda_{2}I)\ldots(A(x)-\lambda_{g}I)=0,$$ for each  $x\in U$. Then  $(A(x)-\lambda_{1} I)(A(x)-\lambda_{2}I)\ldots(A(x)-\lambda_{g}I)Y_{i}(x)=0$, for all  $i~\in~ \{1,\ldots,n\}$ and  $x \in U$. In this way for all
 $x \in U$, $$(A(x)-\lambda_{1} I)X_{i}^{1}(x)=(A(x)-\lambda_{1} I)(A(x)-\lambda_{2}I)\ldots (A(x)-\lambda_{g}I)Y_{i}^{1}(x)=0,$$ for all $i\in\{1,\ldots m_{1}\}$,  $$(A(x)-\lambda_{2} I)X_{i}^{2}(x)=-(A(x)-\lambda_{1} I)(A(x)-\lambda_{2}I)\ldots(A(x)-\lambda_{g}I)Y_{i}^{2}(x)=0,$$ for all $i\in\{m_{1}+1,\ldots,m_{1}+m_{2}\},$ $$(A(x)-\lambda_{k} I)X_{i}^{k}(x)=(-1)^{k-1}(A(x)-\lambda_{1} I)(A(x)-\lambda_{2}I)\ldots(A(x)-\lambda_{g}I)Y_{i}^{k}(x)=0,$$ for all  $i\in\{m_{1}+\cdots+m_{k-1}+1,\ldots,m_{1}+\cdots+m_{k}\}$. Then, for each $x\in U$, we have that $A(x)X_{i}^{k}(x)=\lambda_{k} X_{i}^{k}(x)$,  for all $i\in\{1,\ldots,n\}$ and $k \in \{1,\ldots,g\}$.

  By the  Gram-Schmidt orthogonalization process we obtain orthonormalized sets $\{X_{1}^{1}(x),\linebreak\ldots,X_{m_{1}}^{1}(x)\}$,  $\{X_{m_{1}+1}^{2}(x),\ldots,X_{m_{1}+m_{2}}^{2}(x)\}$ and $\{X_{m_{1}+\cdots+m_{k-1}+1}^{k}(x),\ldots, X_{m_{1}+\cdots+m_{k}}^{k}(x)\}$, for all $x\in U$ and  $k~\in~\{3,\ldots,g\}$. Moreover, let $X_{i}^{\alpha}$ and  $X_{j}^{\beta}$ for  $i,j \in \{1,\ldots,n\}$ and  $\alpha,\beta\in \{1,\ldots,g\}$ with  $\alpha\neq\beta$. Then,
$$
\begin{array}{rcl}
(\lambda_{\alpha}-\lambda_{\beta})\langle X_{i}^{\alpha}, X_{j}^{\beta}\rangle &=& \langle \lambda_{\alpha} X_{i}^{\alpha}, X_{j}^{\beta}\rangle-\langle X_{i}^{\alpha}, \lambda_{\beta} X_{j}^{\beta}\rangle=\langle AX_{i}^{\alpha}, X_{j}^{\beta}\rangle-\langle X_{i}^{\alpha}, AX_{j}^{\beta}\rangle
\\&=& \langle AX_{i}^{\alpha}-AX_{i}^{\alpha}, X_{j}^{\beta}\rangle=0.
\end{array}
$$
Since $ \lambda_{\alpha} \neq \lambda_{\beta}$, for $\alpha \neq \beta$ we get   $\langle X_{i}^{\alpha}, X_{j}^{\beta}\rangle=0$, for $i,j \in \{1,\ldots,n\}$ and $\alpha,\beta\in \{1,\ldots,g\}$,  $\alpha\neq\beta$.
So we obtain a local orthonormal frame of differentiable eigenvalues $\{X_{1},\ldots,X_{n}\}$.
\end{proof}

This  leads  to the following result.

\begin{corollary}Let $\f$, $n\geq 2$, be a hypersurface  having $g\geq2$ distinct principal curvatures $\lambda_{1},\ldots,\lambda_{g}$, with constant multiplicities $m_{1},\ldots,m_{g}$, respectively.
Then for each  $p \in M$, there exist an orthonormal frame of principal directions  $\{X_{1},\ldots, X_{n}\}$ in a neighborhood  $U$ of $p$ in $M$.
\end{corollary}

In the next proposition we obtain some equations that will be useful in this paper.

\begin{proposition}   Let $\f$  be a hypersurface having principal curvatures with  constant multiplicity. Let  $\{X_{1},\ldots,X_{n}\}$ be a frame of principal orthonormal directions and let  $\lambda_{i}$ be the principal curvature associated to  $X_{i}$. If  $T$ is a principal direction, $X_{n}=\|T\|^{-1}T$ and $\eta_{\mathbb{Q}}=\eta-\nu\dt$ then
 \begin{equation}\label{eqobs4}
    \widehat{\nabla}_{X_{i}}\eta_{\mathbb{Q}}=-\lambda_{i}df(X_{i})+c\nu\langle X_{i},T\rangle\xi-X_{i}(\nu)\dt,
    \end{equation}
\begin{equation}\label{eqobs5}
X_{i}(\|T\|)=0,\;\;\textrm{ for all}\;\;i\neq n,\;\;\textrm{ and }\;\;X_{n}(\|T\|)=\nu\lambda_{n},
\end{equation}
 \begin{equation}\label{eqobs6}
X_{i}(\nu)=0,\;\;\textrm{for all  }\;\;i\neq n,\;\;\textrm{ and }\;\;X_{n}(\nu)=-\lambda_{n}\|T\|,
\end{equation}
\begin{equation}\label{eqobs7}
X_{i}(\pi_{2}\circ f)=0,\;\; \textrm{ for all  }\;\; i\neq n\;\;\textrm{ and }\;\;X_{n}(\pi_{2}\circ f)=\|T\|.
\end{equation}

\end{proposition}

\begin{proof} Observe that
 $$\widehat{\nabla}_{X_{i}}\eta_{\mathbb{Q}}=\widehat{\nabla}_{X_{i}}(\eta-\nu\dt)=\widehat{\nabla}_{X_{i}}\eta-X_{i}(\nu)\dt=-df(A_{\eta}X_{i})+
    \nabla^{\perp}_{X_{i}}\eta-X_{i}(\nu)\dt.$$
  By  (\ref{eqobs3}) we get
 $\widehat{\nabla}_{X_{i}}\eta_{\mathbb{Q}}=-\lambda_{i}df(X_{i})+c\nu\langle X_{i},T\rangle\xi-X_{i}(\nu)\dt.$

Using  equations (\ref{eq03}) and  (\ref{eq04}), we get for all $i\in \{1,\ldots,n\}$,
$$2\|T\|X_{i}(\|T\|)=X_{i}(\|T\|^{2})=X_{i}\langle T,T\rangle=2\langle\nabla_{X_{i}}T,T\rangle=2\nu\lambda_{i}\langle X_{i},T\rangle$$
and
$X_{i}(\nu)=-\langle A_{\eta}X_{i}, T\rangle=-\langle X_{i}, A_{\eta}T\rangle=-\lambda_{n}\langle X_{i},T\rangle.$

Moreover,
$$X_{i}(\pi_{2}\circ f)=d\pi_{2}(df(X_{i}))=\pi_{2}df(X_{i})=\langle df(X_{i}),\dt\rangle=\langle X_{i},T\rangle.$$
\end{proof}

\section{Family of parallel hypersurfaces in $\mathbb{S}^{n}\times\mathbb{R}$ and $\mathbb{H}^{n}\times\mathbb{R}$}

Consider the hypersurfaces $\f$ and  $i\colon \QR\rightarrow\E$ with normal fields $\eta$ and $\xi$, respectively such that $\eta$ is unitary and  $\langle\xi,\xi\rangle=c$. Let $F:=i\circ f$, $\pi_{1}\colon\QR~\rightarrow~\mathbb{Q}^{n}_{c}$ and  $\pi_{2}\colon\QR\rightarrow\R$ be the canonical projections. Given  $t\in \R$, $p\in M^{n}$ and  $v\in T_{f(p)}(\QR)$ such that  $d_{f(p)}\pi_{1}(v)=v_{1}$ and  $d_{f(p)}\pi_{2}(v)=v_{2}$, the exponencial map  in $\QR$  is defined by  $$\exp_{f(p)}(tv)=\left(C_{c}(\|v_{1}\|t)\pi_{1}(f(p))+S_{c}(\|v_{1}\|t)\frac{v_{1}}{\|v_{1}\|},\;\pi_{2}(f(p))+tv_{2}\right),\;\;\;\textmd{if}\;\;\; v_{1}\neq0\;\;\; \textrm{and}$$   $$\exp_{f(p)}(tv)=\left(\pi_{1}(f(p)),\;\pi_{2}(f(p))+tv_{2}\right),\;\;\;\textmd{if}\;\;\; v_{1}=0,$$
where
\begin{equation}\label{cos}
C_{c}(s)=\left\{
   \begin{array}{cll}
     \cos(s), & \hbox{if } & c=1\\
     \cosh(s), & \hbox{if} & c=-1
   \end{array}
 \right.
,\;\;\; S_{c}(s)=\left\{
   \begin{array}{cll}
     \sin(s), & \hbox{if} & c=1\\
     \sinh(s), & \hbox{if}& c=-1.
   \end{array}
 \right.
\end{equation}

Take  $p\in M^{n}$,  $v\in T_{f(p)}(\QR)$ and the curve  $\alpha\colon I\subset\mathbb{R}\rightarrow\QR$ given by  $\alpha(t)=\exp_{f(p)}(tv)$. Observe that  $\alpha$ is a geodesic in  $\QR$ that passes through the  point  $\alpha(0)=\left(\pi_{1}(f(p)),\pi_{2}(f(p))\right)=f(p)$ and $\alpha'(0)=\left(v_{1},v_{2}\right)=v$.

From now on  we will study the families of hypersurfaces that are parallel to a hypersurface  having  $T$ as a principal direction. For this, take  $\f$ a hypersurface that has  $T$ as a principal direction and  all the principal curvatures with constant multiplicity.
Let $\{X_{1},\ldots,X_{n}\}$ be a frame of orthonormal  principal directions with  $X_{n}=\|T\|^{-1}T$. Observe that $\xi\circ f=(\pi_{1}\circ f,0)$ and  $\eta_{\mathbb{Q}}=\eta-\nu\dt$ wich implies that  $\|\eta_{\mathbb{Q}}\|=\|T\|\neq0$. Then the  hypersurfaces parallel to  $f$ are given by
\begin{equation}\label{eqpll}
f_{t}=C_{c}(\|T\|t)\xi\circ f+S_{c}(\|T\|t)\|T\|^{-1}\eta_{\mathbb{Q}}+(\pi_{2}\circ f+t\nu)\dt.
\end{equation}

For all  $i\in\{1,\ldots,n\}$ we have
$$\begin{array}{rcl}
df_{t}(X_{i}) & = & -ctS_{c}(\|T\|t)X_{i}(\|T\|)\xi+ C_{c}(\|T\|t)\widehat{\nabla}_{X_{i}}\xi+tC_{c}(\|T\|t)X_{i}(\|T\|)\|T\|^{-1}\eta_{\mathbb{Q}} \\
           &  &  +S_{c}(\|T\|t)\widehat{\nabla}_{X_{i}}\|T\|^{-1}\eta_{\mathbb{Q}}+X_{i}(\pi_{2}\circ f+t\nu)\dt.
 \end{array}
$$
From  (\ref{eqobs1}), (\ref{eqobs4}), (\ref{eqobs5}) and  (\ref{eqobs7}) we obtain
\begin{equation}\label{eqtgi}
df_{t}(X_{i})=\left(\co-\lambda_{i}\|T\|^{-1}\s\right)df(X_{i}),\;\;\textrm{ for}\;\;i\neq n\;\;\;\textrm{ and }
\end{equation}
\begin{equation}\label{eqtgn}
\begin{array}{rcl}
df_{t}(X_{n}) & = & c\nu\s (1-t\lambda_{n})\xi +(1-t\lambda_{n})\left(\nu^{2}\co+\|T\|^{2}\right)df(X_{n})\\
               &  & +(1-t\lambda_{n})\left(1-\co\right)\nu\|T\|\eta.
\end{array}
\end{equation}
Then  $f_{t}$ is an immersion if
$\co-\lambda_{i}\|T\|^{-1}\s\neq0$, for all $i\in\{1,\ldots,n-1\}$ and $1-t\lambda_{n}\neq0$.

Observe  that $\eta_{t}$ given by
\begin{equation}\label{eqnort}
\eta_{t}=-c\|T\|\s\xi\circ f +\co\eta_{\mathbb{Q}}+\nu\dt
\end{equation}
is a unit  vector field normal to $f_{t}$.

Next result gives  the relation between the principal curvatures  of a hypersurface in \; $\QR$ having  $T$ as principal direction and the principal curvatures of its parallel hypersurfaces.

\begin{proposition} Let $\f$ be a hypersurface having $T$ as a principal direction and  $\lambda_{i}$, $i\in\{1,\ldots,n\}$ its principal curvatures. If  $f_{t}$ is a family of hypersurfaces parallel to  $f$ with principal curvatures  $\lambda_{i}^{t}$, $i\in\{1,\ldots,n\}$ then
\begin{equation}\label{eqcti}
\lambda_{i}^{t}=\displaystyle\frac{c\|T\|\s+\lambda_{i}\co}{\co-\lambda_{i}\|T\|^{-1}\s},\;\;\;i\neq n,
\end{equation}
\begin{equation}\label{eqctn}
\textrm{ and }\;\;\;\;\lambda_{n}^{t}=\displaystyle\frac{\lambda_{n}}{1-t\lambda_{n}}.
\end{equation}
\end{proposition}
\begin{proof}
 Let $\{X_{1},\ldots,X_{n}\}$ be an orthogonal frame of principal directions of $f$.

From  (\ref{eqnort}), (\ref{eqobs1}), (\ref{eqobs4}), (\ref{eqobs5}) and (\ref{eqobs6}), we conclude that
 \begin{equation}\label{eqdnti}
 \widehat{\nabla}_{X_{i}}\eta_{t}=-\left(c\|T\|\s+\lambda_{i}\co\right)df(X_{i}),\;\;\;\textrm{ for}\;\;i\neq n \;\;\;\textrm{ and }
  \end{equation}
\begin{equation}\label{eqdntn}
\begin{array}{rcl}
   \widehat{\nabla}_{X_{n}}\eta_{t} & = & \left\{(1-t\lambda_{n})\|T\|\co-\lambda_{n}\s\right\}c\nu\xi \\
   &   & +\left\{-\lambda_{n}(\|T\|^{2}+\nu^{2}\co)-c\nu^{2}\|T\|\s(1-t\lambda_{n})\right\}df(X_{n}) \\
   &   & +\left\{c\nu\|T\|^{2}\s(1-t\lambda_{n})-\nu\|T\|\lambda_{n}(1-\co)\right\}\eta.
\end{array}
\end{equation}
 Observe that if $\{X_{1},\ldots,X_{n}\}$ is an orthogonal frame of principal directions of $f$ then it is also an orthogonal frame of principal directions of $f_{t}$.

From  (\ref{eqtgi}) and  (\ref{eqtgn}) we get
\begin{equation}\label{eqd1}
\langle df_{t}(X_{i}),df_{t}(X_{i})\rangle=\left(\co-\lambda_{i}\|T\|^{-1}\s\right)^{2},\;\;\;\textrm{ for }\;i\neq n\;\;\;\textrm{ and }
\end{equation}
\begin{equation}\label{eqd2}
\langle df_{t}(X_{n}),df_{t}(X_{n})\rangle=(1-t\lambda_{n})^{2}.
\end{equation}
Moreover by using  (\ref{eqtgi}) and  (\ref{eqdnti}) we get for $i\neq n$,
\begin{equation}\label{eqn1}
\begin{array}{rcl}
-\langle \widehat{\nabla}_{X_{i}}\eta_{t},df_{t}(X_{i}) \rangle& = &\left(c\|T\|\s+\lambda_{i}\co\right) \\
&  & \left(\co -\lambda_{i}\|T\|^{-1}\s\right).
\end{array}
\end{equation}
From  (\ref{eqtgn}) and (\ref{eqdntn}) we conclude also that
\begin{equation}\label{eqn2}
 \langle \widehat{\nabla}_{X_{n}}\eta_{t},df_{t}(X_{n}) \rangle=-\lambda_{n}(1-t\lambda_{n}).
\end{equation}
Finally using (\ref{eqd1}) and (\ref{eqn1}) we show (\ref{eqcti}) and from (\ref{eqd2}) and (\ref{eqn2}) we obtain (\ref{eqctn}).
\end{proof}

\begin{remark}
Since we are supposing that  the principal curvatures  $\lambda_{i}$, $i\in\{1,\ldots,n\}$, have constant multiplicity we conclude that  the curvatures  $\lambda_{i}^{t}$, $i\in\{1,\ldots,n\}$, also have constant multiplicity and by  (\cite{Nomizu}) they are differentiable.
\end{remark}

\section{ A necessary and sufficient condition for an isoparametric hypersurface of $\QR$ having constant principal curvatures}

Cartan proved in \cite{Ca1} that a necessary and sufficient condition for a family of parallel hypersurfaces in a Riemannian manifold to be isoparametric is that all hypersurfaces must  have constant mean curvature.

In the next result we obtain the necessary and suficient condition that an isoparametric hypersurface in  $\QR$ having   $T$ as a principal direction must satisfy  to have constant principal curvatures.

\begin{theorem} \label{di}
Let  $\f$ be an isoparametric hypersurface having  $T$ as a principal direction. Then  $f$ has constant principal curvatures if and only if  $\|T\|$ is constant.
\end{theorem}
\begin{proof}
Define the real valued function
$$u(t)=\sum_{i=1}^{n}\lambda_{i}^{t}.$$

 Since  $T$ is a principal direction we may use expressions  (\ref{eqcti}) and (\ref{eqctn}).
Observe that
\begin{equation}\label{eqind1}
\displaystyle\frac{\partial\lambda_{i}^{t} }{\partial t}=c\|T\|^{2}+(\lambda_{i}^{t})^{2},\;\;\;\textrm{for}\;\;i\in\{1,\ldots,n-1\}.
\end{equation}
 If $f$ has constant principal curvatures then $\displaystyle\sum_{i=1}^{n}\lambda_{i}^{k}$, $1\leq k\leq n$ is also constant. We have that $u'(t)=\displaystyle\sum_{i=1}^{n-1}c\|T\|^{2}+(\lambda_{i}^{t})^{2}+(\lambda_{n}^{t})^{2}$. Then
$$u'(0)=(n-1) c\|T\|^{2}+\displaystyle\sum_{i=1}^{n}\lambda_{i}^{2}$$
and $\|T\|$ is constant.

On the converse if $\|T\|$ is constant  the function $\nu$ is constant and by
(\ref{eq04}), $\lambda_{n}=0$. Hence from  (\ref{eqctn}), it follows that $\lambda_{n}^{t}=0$.
 So $u(t)=\displaystyle\sum_{i=1}^{n-1}\lambda_{i}^{t}$ and consequently its derivative of order  $k$ is  $u^{k}(t)=\displaystyle\sum_{i=1}^{n-1}\frac{\partial^{k}\lambda_{i}^{t} }{\partial t^{k}}$.
Observe that
\begin{equation}\label{eqind2}
\displaystyle\frac{\partial^{2}\lambda_{i}^{t} }{\partial t^{2}}=2c\|T\|^{2}\lambda_{i}^{t}+2(\lambda_{i}^{t})^{3}.
\end{equation}
Let us prove using the  induction process that for odd $k$,  $2< k< n$  we get,
\begin{equation}\label{eqki}
\begin{array}{rcl}
 \displaystyle\frac{\partial^{k}\lambda_{i}^{t} }{\partial t^{k}} & = & u_{k,0}c^{\frac{k+1}{2}}\|T\|^{k+1}+u_{k,2}c^{\frac{k-1}{2}}\|T\|^{k-1}(\lambda_{i}^{t})^{2}+u_{k,4}c^{\frac{k-3}{2}}\|T\|^{k-3}(\lambda_{i}^{t})^{4}\\
 &   & +\ldots+u_{k,k+1}(\lambda_{i}^{t})^{k+1},
\end{array}
\end{equation}
where we denote by $u_{k,j}$ the $j$-th coefficient $u_{j}$ of the derivative of order  $k$ of  $\lambda_{i}^{t}$.

If   $k$ is even,  $2\leq k< n$  we obtain
\begin{equation}\label{eqkp}
\begin{array}{rcl}
\displaystyle\frac{\partial^{k}\lambda_{i}^{t} }{\partial t^{k}}&=&u_{k,1}c^{\frac{k}{2}}\|T\|^{k}\lambda_{i}^{t}+u_{k,3}c^{\frac{k-2}{2}}\|T\|^{k-2}(\lambda_{i}^{t})^{3}+u_{k,5}c^{\frac{k-4}{2}}\|T\|^{k-4}(\lambda_{i}^{t})^{5}
+\ldots \\
&  &+u_{k,k+1}(\lambda_{i}^{t})^{k+1},
\end{array}
\end{equation}
where $u_{k,0}=u_{k-1,1}$, \,\,$u_{k,1}=2u_{k-1,2}$,\,\, $u_{k,2}=3u_{k-1,3}+u_{k-1,1},\;\ldots\;,$\,\,
 $u_{k,j}=(j+1)u_{k-1,j+1}+(j-1)u_{k-1,j-1},\;\ldots ,$ \; $u_{k,k+1}=ku_{k-1,k}$. We point out  that if  $k$ is odd  the index  $j$ of $u_{k,j}$ is an even number and when   $k$ is even the index  $j$ of $u_{k,j}$ is odd.

By (\ref{eqind1}), we get $u_{1,2}=1$. For $k=2$ and using equation  (\ref{eqind2}), we obtain
$$\displaystyle\frac{\partial^{2}\lambda_{i}^{t} }{\partial t^{2}}=2c\|T\|^{2}\lambda_{i}^{t}+2(\lambda_{i}^{t})^{3}=2u_{1,2}c \|T\|^{2}\lambda_{i}^{t}+2u_{1,2}(\lambda_{i}^{t})^{3}=u_{2,1} c \|T\|^{2}\lambda_{i}^{t}+u_{2,3}(\lambda_{i}^{t})^{3},$$
that satisfies equation (\ref{eqkp}).

By the induction hypothesis let us suppose that equations  (\ref{eqki}) and  (\ref{eqkp}) hold for the index $k-1$. We will show that they  hold also  for the index  $k$.

If  $k$ is an even number  then   $k-1$ is an odd number  and equation  (\ref{eqki}) holds, that is,
\begin{equation} \label{eqh}
\begin{array}{rcl}
{\displaystyle \frac{\partial^{k-1}\lambda_{i}^{t} }{\partial t^{k-1}} } &= & u_{k-1,0}c^{\frac{k}{2}}\|T\|^{k}+u_{k-1,2}c^{\frac{k-2}{2}}\|T\|^{k-2}(\lambda_{i}^{t})^{2}\\
 &  & +u_{k-1,4}c^{\frac{k-4}{2}}\|T\|^{k-4}(\lambda_{i}^{t})^{4} +\ldots+u_{k-1,k}(\lambda_{i}^{t})^{k}.
\end{array}
\end{equation}
By deriving equation  (\ref{eqh}), with respect to the variable  $t$, using equation (\ref{eqind1}), we get
$$
\begin{array}{rcl}
  \displaystyle\frac{\partial^{k}\lambda_{i}^{t} }{\partial t^{k}} & = & 2u_{k-1,2}c^{\frac{k-2}{2}}\|T\|^{k-2}\lambda_{i}^{t}\displaystyle\frac{\partial\lambda_{i}^{t} }{\partial t}+4u_{k-1,4}c^{\frac{k-4}{2}}\|T\|^{k-4}(\lambda_{i}^{t})^{3}\displaystyle\frac{\partial\lambda_{i}^{t} }{\partial t}+\ldots\\
   &   & +ku_{k-1,k}(\lambda_{i}^{t})^{k-1}\displaystyle\frac{\partial\lambda_{i}^{t} }{\partial t} \\
   & = &  2u_{k-1,2}c^{\frac{k}{2}}\|T\|^{k}\lambda_{i}^{t}+(2u_{k-1,2}+4u_{k-1,4})c^{\frac{k-2}{2}}\|T\|^{k-2}(\lambda_{i}^{t})^{3}+\ldots\\
   &   &+ ku_{k-1,k}(\lambda_{i}^{t})^{k+1}\\
   & = & u_{k,1}c^{\frac{k}{2}}\|T\|^{k}\lambda_{i}^{t}+u_{k,3}c^{\frac{k-2}{2}}\|T\|^{k-2}(\lambda_{i}^{t})^{3}+\ldots+
   u_{k,k+1}(\lambda_{i}^{t})^{k+1}.
\end{array}
$$
Then equation  (\ref{eqkp}) holds. In the same way it can be shown that  equation  (\ref{eqki}) also holds.

Since $\lambda_{n}=0$, we have that $u(0)=\displaystyle\sum_{i=1}^{n-1}\lambda_{i}=C_{1}$, with $C_{1}$ constant and
$$u'(0)=\displaystyle\sum_{i=1}^{n-1}c\|T\|^{2}+\lambda_{i}^{2},$$ which implies that $\displaystyle\sum_{i=1}^{n}\lambda_{i}^{2}=C_{2}$, with $C_{2}$ constant.
Hence $$u''(0)=(n-1)2c\|T\|^{2}\displaystyle\sum_{i=1}^{n-1}\lambda_{i}+2\displaystyle\sum_{i=1}^{n-1}\lambda_{i}^{3}=(n-1)2c\|T\|^{2}C_{1}+
2\displaystyle\sum_{i=1}^{n-1}\lambda_{i}^{3},$$ and
$\displaystyle\sum_{i=1}^{n}\lambda_{i}^{3}=C_{3}$, with $C_{3}$ constant.

If $k$ is an even number  we get
$$\begin{array}{rcl}
u^{k}(0) & = & u_{k,1}c^{\frac{k}{2}}\|T\|^{k}\displaystyle\sum_{i=1}^{n-1}\lambda_{i}+
u_{k,3}c^{\frac{k-2}{2}}\|T\|^{k-2}\displaystyle\sum_{i=1}^{n-1}(\lambda_{i})^{3}+u_{k,5}c^{\frac{k-4}{2}}\|T\|^{k-4}\displaystyle\sum_{i=1}^{n-1}(\lambda_{i})^{5}\\
         &   & +\ldots+u_{k,k+1}\displaystyle\sum_{i=1}^{n-1}(\lambda_{i})^{k+1}\\
    & = & u_{k,1}c^{\frac{k}{2}}\|T\|^{k}C_{1}+ u_{k,3}c^{\frac{k-2}{2}}\|T\|^{k-2}C_{3}+u_{k,5}c^{\frac{k-4}{2}}\|T\|^{k-4}C_{5}
    +\ldots \\
      &  & + u_{k,k+1}\displaystyle\sum_{i=1}^{n-1}(\lambda_{i})^{k+1}.
\end{array}
$$
Then $\displaystyle\sum_{i=1}^{n}\lambda_{i}^{k+1}$ is constant. In a similar way we obtain the same if  $k$ is an odd number.
Finally  we conclude that $\displaystyle\sum_{i=1}^{n}\lambda_{i}^{k}$, $1\leq k\leq n$, is constant.
 Based on the demonstration of \;\cite[Theorem 5.8]{CR}, by Newton identity the coeficients of the characteristic polynomial  of the  Weingarten  operator $A$ are also polynomials  $\displaystyle\sum_{i=1}^{n}\lambda_{i}^{k}$, $1\leq k\leq n$. Then the principal curvatures are also constant because they are the roots of the characteristic polynomial.
 \end{proof}

Next result for a hypersurface with constant angle $\nu\neq1$, given by \cite[Corollary 2]{RT1}, is proved using Theorem \ref{di}
\begin{corollary} \label{haci}
Let  $\f$ be a hypersurface with constant angle and  $T\neq0$. Then  $f$ is  isoparametric if and only if the principal curvatures are constant.
\end{corollary}
\begin{proof}
By Remark \ref{obsvctd}, $T$ is a principal direction.
 Since $f$ has $\nu$ constant  from  (\ref{vt1})  $\|T\|$ is also constant.  Moreover by  (\ref{eq04}), $\lambda_{n}=0$ and so by  (\ref{eqctn}), $\lambda_{n}^{t}=0$.

 Suppose that the principal curvatures  $\lambda_{i}$ of  $f$ are constant. By (\ref{eqcti}) and  (\ref{eqctn}) the principal curvatures of the family   $f_{t}$ are also constant and  $f$ is isoparametric.

On the converse if $f$ is  isoparametric, since  $\|T\|$ is constant, by Theorem  \ref{di}, the principal curvatures of  $f$ are  constant.
\end{proof}

\section{Hypersurfaces in $\QR$ having  constant principal curvatures and  $T$ as a principal direction}

In this section we classify the hypersurfaces of $\QR$ with constant principal curvatures having field $T$ as a principal direction. First, we state  the following technical lemma.

\begin{lemma}\label{lem01} Let  $a:I\subset\mathbb{R}\rightarrow\mathbb{R}$ be a differentiable function such that  $a'(s)>0$ and  $a''(s)\neq0$, for all  $s\in I$. The solutions of the differential equation $a'''(1+(a')^{2})-3(a'')^{2}a'~=~0$ are given by  $$a(s)=-\frac{\sqrt{1-(c_{1}s+c_{2})^{2}}}{c_{1}}+c_{3},$$ where  $c_{1}, c_{2}$ and  $c_{3}$ are real constant, $0<c_{1}s+c_{2}<1$ and  $c_{1}\neq0$.
\end{lemma}

\begin{theorem}\label{ctdp}Let  $\f$, $n\geq2$ be a hypersurface having  constant principal curvatures and   $T$ as a principal direction  such that  $\nu(p)\neq0$, for all $p\in M$. Then  $c=-1$ and  $f$ is given locally by  $f(x,s)= h_{s}(x)+Bs\dt$, for some  $B\in \mathbb{R}$, $B>0$, where  $h_{s}$ is a family of horospheres in  $\mathbb{H}^{n}$. Moreover the principal curvature  associated to the field  $T$ is equal to $0$ and the others principal curvatures are all equal to $\displaystyle\frac{B}{\sqrt{1+B^{2}}}$ or  $-\displaystyle\frac{B}{\sqrt{1+B^{2}}}$.
\end{theorem}
\begin{proof}  By \cite[Theorem 1]{RT1} if  $T$ is a principal direction of  $f$ and  $\nu(p)\neq0$, for all $p\in M$ then  $f$ is locally given by
 $f:M^{n}=M^{n-1}\times I\rightarrow \QR$  with  $f(x,s)~=~ h_{s}(x)+a(s)\dt$, where  $a:I\rightarrow \R$ is a  differentiable function  such that  $a'(s)>0$,  for all   $s\in I$. Moreover $A_{\eta}X=-\displaystyle\frac{a'(s)}{b(s)}A^{s}X$, for all  $X\in TM^{n-1}$, where $A^{s}$ is the shape operator of $h_{s}$. In particular,  $A_{\eta}X_{i}=-\displaystyle\frac{a'(s)}{b(s)}\lambda_{i}^{s}(x)X_{i}$, for the principal directions  $X_{i}\in TM^{n-1}$ of  $h$ and  $A_{\eta}T=\displaystyle\frac{a''(s)}{b^{3}(s)}T$, where  $b(s)=\sqrt{1+a'(s)^{2}}$. Therefore $$A_{\eta}(X_{i})=\mu_{i}(x,s)X_{i},\;\;\;\textrm{ with} \;\;\;\mu_{i}(x,s)=-\displaystyle\frac{a'(s)}{b(s)}\lambda_{i}^{s}(x),$$ for $i\in\{1,...,n-1\}$ and  $A_{\eta}T=\mu_{n}(x,s)T$, with  $\mu_{n}(x,s)=\displaystyle\frac{a''(s)}{b^{3}(s)}$.

 It is known that the relation between the principal curvatures of a hypersurface of $\mathbb{Q}_{c}^{n}$ and the principal curvatures of its
parallel hypersurfaces is given by
$$\lambda_{i}^{s}(x)=\frac{cS_{c}(s)+C_{c}(s)\lambda_{i}(x)}{C_{c}(s)-S_{c}(s)\lambda_{i}(x)}$$
 and therefore
 $$\mu_{i}(x,s)=-\frac{a'(s)}{b(s)}\left(\frac{cS_{c}(s)+C_{c}(s)\lambda_{i}(x)}{C_{c}(s)-S_{c}(s)\lambda_{i}(x)}\right), \;\;\;i \in \{1,...,n-1\}.$$

 Let us analyze under which conditions  the functions $\mu_{i}$ are constant.
Observe that
 \begin{equation}\label{3.4}
 \displaystyle\frac{\partial\lambda_{i}^{s}}{\partial x}=\displaystyle\frac{\lambda'_{i}(x)}{(C_{c}(s)-S_{c}(s)\lambda_{i}(x))^{2}}
 \end{equation}
 and
 \begin{equation}\label{3.5}
 \displaystyle\frac{\partial\lambda_{i}^{s}}{\partial s}=c+(\lambda_{i}^{s}(x))^{2}.
 \end{equation}

For $i\in \{1,...,n-1\}$,
$\displaystyle\frac{\partial\mu_{i}}{\partial x}=-\displaystyle\frac{a'(s)}{b(s)} \displaystyle\frac{\partial\lambda_{i}^{s}}{\partial x}$  and $\displaystyle\frac{\partial\mu_{i}}{\partial x}=0$  if and only if $\displaystyle\frac{\partial\lambda_{i}^{s}}{\partial x}=0$.

Therefore,  by (\ref{3.4}),  $\mu_{i}$ is  constant with respect to  $x$ if   $h$ is  isoparametric,  that is,  if the  functions $\lambda_{i}$ are constant for all  $i\in \{1,...,n-1\}$.

Moreover,
$$ \frac{\partial\mu_{i}}{\partial s}  =  \frac{(-a''b+a'b')}{b^{2}}\lambda^{s}_{i}-\frac{a'}{b}\frac{\partial\lambda_{i}^{s}}{\partial s}
      =  -\frac{a''}{b^{3}}\lambda^{s}_{i}-\frac{a'}{b}(c+(\lambda_{i}^{s})^{2})
      =  \frac{-a''\lambda^{s}_{i}-a'b^{2}(c+(\lambda^{s}_{i})^{2})}{b^{3}}.$$

Then, for $i\in\{1,...,n-1\}$, $\displaystyle\frac{\partial\mu_{i}}{\partial s} =0$ if and only if
 $a''\lambda^{s}_{i}+a'(1+a'^{2})(c+(\lambda^{s}_{i})^{2})=0$.

Now let us analyze the  curvature $\mu_{n}(x,s)=\displaystyle\frac{a''(s)}{b^{3}(s)}$. We have $\displaystyle\frac{\partial\mu_{n}}{\partial x}=0$ and
$$\frac{\partial\mu_{n}}{\partial s}=\frac{a'''b^{3}-3a''b^{2}b'}{b^{6}}=\frac{a'''b-3a''b'}{b^{4}}.$$

Consequently, $\displaystyle\frac{\partial\mu_{n}}{\partial s}=0$ if and only if $a'''b-3a''b'=0$, that is, $ a'''(1+a'^{2})~-~3a''^{2}a'=0$.

Consider for all $s\in I$ and $i\in\{1,...,n-1\}$, the following  equations,
\begin{equation}\label{eq2} a''\lambda^{s}_{i}+a'(1+a'^{2})(c+(\lambda^{s}_{i})^{2})=0\;\;\;\;\;\;\textrm{ and }
\end{equation}
\begin{equation}\label{eq3} a'''(1+a'^{2})-3a''^{2}a'=0.
\end{equation}
As the  function  $a$ is differentiable of  class  $\mathcal{C}^{\infty}$ we may consider just two cases: $a''(s)=0$, for all  $s\in I$ or  $a''(s)\neq0$ for all  $s\in I$,  restricting the interval  $I\in \mathbb{R}$, if necessary.

$\textbf{Case 1}$: Suppose that  $a''(s)=0$ for all $s\in I$. Then  $a'''(s)=0$ and  consequently  the  equation  (\ref{eq3}) holds.
 By equation  (\ref{eq2}) we get  $a'(1+a'^{2})(c+(\lambda^{s}_{i})^{2})=0$. Since  $a'(s)>0$ we conclude that  $c+(\lambda^{s}_{i})^{2}=0$.

 If  $c=1$ then  $1+(\lambda^{s}_{i})^{2}\neq0$. So this case cannot occur for  $c=1$.

 If  $c=-1$ then  $(\lambda^{s}_{i})^{2}=1$. So $\lambda^{s}_{i}=\pm 1$ which implies that $\lambda_{i}=\pm 1$. Moreover, $\mu_{n}=0$ and  $\mu_{i}=\pm\displaystyle\frac{a'}{b}$, for  $i\in \{1,...,n-1\}$.

$\textbf{Case 2}$: Suppose that  $a''(s)\neq0$ for all  $s\in I$. From  Lemma \ref{lem01} the  solutions of the  equation  (\ref{eq3}) are given by  $a(s)=-\displaystyle\frac{\sqrt{1-(c_{1}s+c_{2})^{2}}}{c_{1}}+c_{3}$, where $c_{1}, c_{2}$ and $c_{3}$ are real constant with $c_{1}\neq0$. Let us certify if those solutions  satisfy the  equation  (\ref{eq2}). Observe that

$$a'(s)=\displaystyle\frac{c_{1}s+c_{2}}{\sqrt{1-(c_{1}s+c_{2})^{2}}}\;\;\;\;\;\;\; \textrm{ and } \;\;\;\;\;\;\;a''(s)  = \displaystyle\frac{c_{1}}{(1-(c_{1}s+c_{2})^{2})^{\frac{3}{2}}}. $$
Then
$$
   a'(1+a'^{2}) = \displaystyle\frac{c_{1}s+c_{2}}{(1-(c_{1}s+c_{2})^{2})^{\frac{3}{2}}}.
$$
Thus,
$$a''\lambda^{s}_{i}+a'(1+a'^{2})(c+(\lambda^{s}_{i})^{2})= \displaystyle\frac{c_{1}}{(1-(c_{1}s+c_{2})^{2})^{\frac{3}{2}}}\lambda_{i}^{s}+\displaystyle\frac{c_{1}s+c_{2}}{(1-(c_{1}s+c_{2})^{2})^{\frac{3}{2}}}(c+(\lambda^{s}_{i})^{2}).$$

Then the solutions  $a(s)$ of the equation  (\ref{eq3}) satisfy  (\ref{eq2})  if and only if
\begin{equation}\label{eq4}c_{1}\lambda_{i}^{s}+(c_{1}s+c_{2})(c+(\lambda^{s}_{i})^{2})=0,\end{equation} for all $s\in I$.

Suppose that (\ref{eq4}) holds, for all  $s\in I$. If  $c+(\lambda^{s}_{i})^{2}=0$ then  $c_{1}\lambda_{i}^{s}=0$, i.e., $\lambda_{i}^{s}=0$, for all $s\in I$, since $c_{1}\neq 0$. But,
$$\lambda_{i}^{s}(x)=\frac{cS_{c}(s)+C_{c}(s)\lambda_{i}(x)}{C_{c}(s)-S_{c}(s)\lambda_{i}(x)}=0\;\; \textrm{implies that }\;\; cS_{c}(s)+C_{c}(s)\lambda_{i}(x)=0,$$
for all  $s\in I$, which cannot occur. Then, in this case, it is not possible to have \\ $c+(\lambda^{s}_{i})^{2}~=~0$.

Moreover, by deriving  equation (\ref{eq4}) we obtain  $$\displaystyle\frac{\partial }{\partial s}(c_{1}\lambda_{i}^{s}+(c_{1}s+c_{2})(c+(\lambda^{s}_{i})^{2}))=c_{1}\displaystyle\frac{\partial \lambda_{i}^{s}}{\partial s}+ c_{1}(c+(\lambda^{s}_{i})^{2})+2(c_{1}s+c_{2})\lambda_{i}^{s} \displaystyle\frac{\partial \lambda_{i}^{s}}{\partial s}=0.$$
By  equation  (\ref{3.5}), it follows that
$$2c_{1}(c+(\lambda^{s}_{i})^{2})+2(c_{1}s+c_{2})\lambda_{i}^{s}(c+(\lambda^{s}_{i})^{2}) = 2(c+(\lambda^{s}_{i})^{2})\left(c_{1}+(c_{1}s+c_{2})\lambda_{i}^{s}\right)=0.
$$

Consequently, $c_{1}+(c_{1}s+c_{2})\lambda_{i}^{s}=0$.
Since   $c_{1}s+c_{2}>0$, for all  $s\in I$ we have $\lambda_{i}^{s}=\displaystyle\frac{-c_{1}}{c_{1}s+c_{2}}$.
Thus,  in one way, $$\displaystyle\frac{\partial \lambda_{i}^{s}}{\partial s}=\frac{(c_{1})^{2}}{(c_{1}s+c_{2})^{2}}=(\lambda_{i}^{s})^{2}.$$
But, on the other way, from  (\ref{3.5}) we have that  $\displaystyle\frac{\partial \lambda_{i}^{s}}{\partial s}=c+(\lambda_{i}^{s})^{2}$. Consequently  $c=0$ which cannot occur since we are considering only  $c=1$ or $c=-1$.
So equation  (\ref{eq4})  does not hold and the solutions of the equation  (\ref{eq3}) are not solutions of the equation  (\ref{eq2}), with the condition   $a''(s)\neq0$, for all  $s\in I$.

Thus we conclude that  $a''=0$ and  $a(s)=Bs$ with  $B\in \mathbb{R}$, $B>0$, since  $a'(s)>0$ and  $\nu$ is constant, from  \cite[Corollary 2]{RT1}.
Therefore $\lambda=\mu_{n}=0$ and $\mu_{i}=\displaystyle\frac{B}{\sqrt{1+B^{2}}}$ or $\mu_{i}=-\displaystyle\frac{B}{\sqrt{1+B^{2}}}$, for $i\in \{1,\ldots,n-1\}$.
\end{proof}

\begin{remark}\label{obsrot}
 From \cite[Proposition 20]{Manfio}, for $n\geq3$, and \cite[Remark 7 (i)]{RT1}, for $n=2$, the hypersurfaces  given in  Theorem \ref{ctdp}, are rotational hypersurfaces in  $\mathbb{H}^{n}\times\mathbb{R}$ for which the orbits are horospheres.
\end{remark}

\section{Multiplicities of the principal curvatures}\label{sec:umbilicity results}

In this section we discuss  some results about the multiplicities of the principal curvatures of hypersurfaces in $\QR$.
\begin{theorem}\label{m1} Let  $\f$, $n\geq 2$ be a  non umbilical hypersurface having  constant principal curvatures with  constant  multiplicities and suppose that its function $\nu\neq 0$. Then it has at least one principal  curvature of multiplicity one.
\end{theorem}
\begin{proof}Let  $\{X_{1}, X_{2}, ..., X_{n}\}$ be a local orthonormal frame field of principal directions of  $f$. It is possible to write  $T=\displaystyle\sum_{i=1}^{n}b_{i}X_{i}$. As   $g$ is  the number of distinct principal curvatures and   $f$ is non umbilical then $g\geq 2$.

If  $n=2$ then there exist  two distinct principal curvatures  and each one has multiplicity equal to $1$.

If  $n=3$ then  $g=2$ or $g=3$. If  $g=2$ one of the curvatures has multiplicity  equal to $2$ and the other one  has multiplicity equal to $1$. If  $g=3$ each curvature has multiplicity equal to $1$.

If  $n\geq 4$ suppose that all the principal curvatures have  multiplicity greater than or equal to $2$. In this case  $2\leq g\leq \displaystyle\frac{n}{2}$, if  $n$ is an even number and   $2\leq g\leq\displaystyle\frac{n-1}{2}$ if  $n$ is an odd number. For a given  $\rho\in \{1,...,g\}$ consider $B_{\rho}=\{i\in\{1,...,n\}/ AX_{i}=\lambda_{\rho}X_{i}\}$. Observe that  $B_{\rho}$ has  at least two elements.

For a given  $\rho$, consider the Codazzi equation,
\begin{equation}\label{eq0}
\nabla_{X_{i}}AX_{j}- \nabla_{X_{j}}AX_{i}-A[X_{i},X_{j}]=c\nu (b_{j}X_{i}-b_{i}X_{j}),
\end{equation}
for  $i,j\in B_{\rho}$.

We have  \; $\nabla_{X_{i}}AX_{j}- \nabla_{X_{j}}AX_{i}=\lambda_{\rho}[X_{i},X_{j}]=\lambda_{\rho}\displaystyle\sum_{k=1}^{n}\langle [X_{i},X_{j}],X_{k}\rangle X_{k}$\; and  \linebreak $A[X_{i},X_{j}]=\displaystyle\sum_{k=1}^{n}\langle [X_{i},X_{j}],X_{k}\rangle AX_{k}$. Thus,
\begin{equation}\label{eq3.4.10}
\nabla_{X_{i}}AX_{j}- \nabla_{X_{j}}AX_{i}-A[X_{i},X_{j}]=\sum_{k\not\in B_{\rho}}\langle [X_{i},X_{j}],X_{k}\rangle(\lambda_{\rho}X_{k}- AX_{k}).
\end{equation}
From equations  (\ref{eq0}) and  (\ref{eq3.4.10}),  we get
$$\sum_{k\not\in B_{\rho}}\langle [X_{i},X_{j}],X_{k}\rangle(\lambda_{\rho}-\lambda_{k})X_{k}-c\nu b_{j}X_{i}+c\nu b_{i}X_{j}=0.$$
Since  $i,j\in B_{\rho}$ with  $i\neq j$,  $k\not\in B_{\rho}$ and  $\nu\neq0$ we should have  $b_{i}=b_{j}=0$ for all  $i,j\in B_{\rho}$, that is, $T$ does not have components in the directions corresponding to principal curvatures whose multiplicities are greater than or equal to $2$. So, assuming that there does not exist principal curvatures whose multiplicities are one,  we conclude that  $T=0$.  Finally we conclude that   $f(M^{n})$ is an open subset of a slice   $\mathbb{Q}^{n}_{c}\times \{t\}$ and thus is totally geodesic.  But this is against the hypothesis  $g\geq 2$.  So $f$ has at least one principal curvature  with multiplicity  one.
\end{proof}

\begin{remark}\label{obs31}
From the proof of the previous theorem  we infer that $T$ has no components  in the  directions whose correspondent curvatures  have multiplicities greater than or equal to $2$.
\end{remark}

\begin{remark} Theorem \ref{m1} holds also for   $\nu\equiv0$ if  $c=-1$ since the corresponding curvature of the factor  $\mathbb{R}$ is  $\lambda=0$  and the others curvatures are non zero by   \cite[Theorem 5]{DF}. It is true also for  $\nu\equiv0$ and  $c=1$ if  $g=2$ and $g=3$, excluding the cases  when  $f(M^{n})$ is an open set of  $M^{n-1}\times\mathbb{R}$ where  $M^{n-1}$ is a Cartan's hypersurface for $n\in\{7,13,25\}$.
\end{remark}

\begin{proposition}\label{dp} Let  $\f$, $n\geq 3$ be a hypersurface with  constant principal curvatures and respective multiplicities also constant having   function $\nu \neq 0$. If just one principal curvature has multiplicity  equal to one then the vector field  $T$ is a principal direction corresponding to that curvature. Moreover,  all the curvatures  having multiplicity greater than one do not vanish.
\end{proposition}

\begin{proof} Let $\{X_{1}, X_{2}, ..., X_{n}\}$ be a local orthonormal frame of principal directions of  $f$. Suppose, without loss of generality, that  $X_{n}$ is associated to  $\lambda$, i.e., $AX_{n}=\lambda X_{n}$. By Remark \ref{obs31} if  $\lambda$ is the only curvature  of multiplicity one then   $T=bX_{n}$ where  $b:U~\subset~ M^{n}\rightarrow\mathbb{R}$ is a differentiable function defined on an open subset  $U\subset M^{n}$, where the fields  $X_{1}, X_{2}, ..., X_{n}$ are defined.

From the hypothesis we know that $g\geq2$. For a given  $\rho\in \{1,...,g-1\}$ let  $B_{\rho}=\{i\in\{1,...,n\}/ AX_{i}=\mu_{\rho}X_{i}\}$ with  $\mu_{\rho}\neq\lambda=\mu_{g}$. Observe that $B_{\rho}$ has at least two elements.  Consider the   Codazzi equation
\begin{equation}\label{eq3.4.51}
 \nabla_{X_{n}}AX_{i}- \nabla_{X_{i}}AX_{n}-A[X_{n},X_{i}]=c\nu (b_{i}X_{n}-b_{n}X_{i}),
\end{equation}
for $i\in B_{\rho}$.
We have
\begin{equation}\label{eq3.4.52}
\nabla_{X_{n}}AX_{i}- \nabla_{X_{i}}AX_{n}=\mu_{\rho}\nabla_{X_{n}}X_{i}- \lambda\nabla_{X_{i}}X_{n}\;\;\;\textrm{ and }\;\;\;
\end{equation}
\begin{equation}\label{eq3.4.53}
 A[X_{n},X_{i}]=\sum_{k=1}^{n}\langle \nabla_{X_{n}}X_{i}-\nabla_{X_{i}}X_{n}, X_{k}\rangle AX_{k}.
\end{equation}
Therefore from equations  (\ref{eq3.4.52}) and  (\ref{eq3.4.53}) we get
$$\begin{array}{rcl}
\nabla_{X_{n}}AX_{i}- \nabla_{X_{i}}AX_{n}-A[X_{n},X_{i}] &=&\displaystyle\sum_{k=1}^{n}\langle \nabla_{X_{n}}X_{i},X_{k}\rangle(\mu_{\rho}X_{k}-AX_{k}) \\
& & -\displaystyle\sum_{k=1}^{n}\langle\nabla_{X_{i}}X_{n}, X_{k}\rangle (\lambda X_{k}-AX_{k}).
\end{array}$$
Now using   (\ref{eq3.4.51}) we get
$$\sum_{k\not\in B_{\rho}}\langle \nabla_{X_{n}}X_{i},X_{k}\rangle(\mu_{\rho}-\mu_{k})X_{k} -\sum_{k\neq n}\langle\nabla_{X_{i}}X_{n}, X_{k}\rangle (\lambda-\mu_{k} )X_{k}-c\nu b_{i}X_{n}+c\nu b_{n}X_{i}=0.$$

Then
\begin{equation}\label{eq1}
c\nu b-\langle\nabla_{X_{i}}X_{n}, X_{i}\rangle (\lambda-\mu_{\rho})=0,\;\;\;\forall i\in B_{\rho},
\end{equation}
and $b=0\Leftrightarrow \langle\nabla_{X_{i}}X_{n}, X_{i}\rangle=0$ for $i\in B_{\rho}$.

 By (\ref{eq03}),
$\nu\mu_{\rho}X_{i}=\nabla_{X_{i}}T=\nabla_{X_{i}}bX_{n}=X_{i}(b)X_{n}+b\nabla_{X_{i}}X_{n},\;i\in \beta_{\rho},$
which implies that  $\nu\mu_{\rho}\langle X_{i},X_{i}\rangle=X_{i}(b)\langle X_{n},X_{i}\rangle+b\langle\nabla_{X_{i}}X_{n},X_{i}\rangle$,
i.e., $b\langle\nabla_{X_{i}}X_{n},X_{i}\rangle=\nu\mu_{\rho}$.
Then  $c\nu b\langle\nabla_{X_{i}}X_{n},X_{i}\rangle=c\nu^{2}\mu_{\rho}$
and by equation  (\ref{eq1}), $$\langle\nabla_{X_{i}}X_{n}, X_{i}\rangle^{2} (\lambda-\mu_{\rho})=c\nu^{2}\mu_{\rho}.$$
Consequently
$$b=0\Leftrightarrow\langle\nabla_{X_{i}}X_{n}, X_{i}\rangle=0,\;\;\;\forall i \in B_{\rho}\Leftrightarrow \mu_{\rho}=0.$$

If  $\mu_{\rho}=0$ then  $b=0$ and  $T=0$, that is, $f(M)$ is totally geodesic, which is impossible since $g\geq2$. Then  $T$ is a  principal direction and all the curvatures of multiplicity greater than one do not vanish.  \end{proof}

From  Theorem \ref{ctdp} we obtain the converse of  Proposition \ref{dp} and next result is also true.

\begin{proposition} Let   $\f$, $n\geq 3$ with function  $\nu\neq 0$ be a hypersurface having constant principal curvatures with constant  multiplicities. Then the vector field  $T$ is a principal direction  if and only if there exist only one principal curvature of multiplicity one. Moreover, all the curvatures of multiplicity greater than one do not vanish.
\end{proposition}

\section{Hypersurfaces of $\QR$ with constant principal curvatures for $g\in\{1,2,3\}$}

 In this section we present a result that classifies hypersurfaces with constant principal curvatures. For this we need some propositions.

\begin{proposition}\label{prop21} Let $f:M^{2}\rightarrow\mathbb{Q}^{2}_{c}\times\mathbb{R}$ be a surface with two distinct constant principal curvatures  $\lambda_{1}$ and $\lambda_{2}$. Let  $\{X_{1}, X_{2}\}$ be an orthonormal frame of principal directions corresponding to  $\lambda_{1}$ and $\lambda_{2}$. Consider $T=b_{1}X_{1}+b_{2}X_{2}$, where $b_{1},b_{2}:M^{2}\rightarrow \mathbb{R}$ are differentiable functions. Then
 \begin{equation} \label{eq2.21}
 \lambda_{1}\lambda_{2}+2c\nu^{2}+\displaystyle\frac{c(\lambda_{1}b_{1}^{2}-\lambda_{2}b_{2}^{2})}{\lambda_{2}-\lambda_{1}}
+\displaystyle\frac{2\nu^{2}(b_{1}^{2}+b_{2}^{2})}{(\lambda_{2}-\lambda_{1})^{2}}=0.
 \end{equation}
 \end{proposition}

\begin{proof}
From  Codazzi equation we get
\begin{equation}\label{eq4.1}\nabla_{X_{1}}AX_{2}- \nabla_{X_{2}}AX_{1}-A[X_{1},X_{2}]=c\nu (b_{2}X_{1}-b_{1}X_{2}).
\end{equation}

Observe that
$$
\nabla_{X_{1}}AX_{2}=\lambda_{2}\nabla_{X_{1}}X_{2}=\lambda_{2}\langle\nabla_{X_{1}}X_{2},X_{1}\rangle X_{1},
$$
since  $X_{1}\langle X_{2},X_{2}\rangle=0$.
In a similar way we get
$$
\nabla_{X_{2}}AX_{1}=\lambda_{1}\langle\nabla_{X_{2}}X_{1},X_{2}\rangle X_{2}.
$$
Thus,
$$
A[X_{1},X_{2}]=A(\nabla_{X_{1}}X_{2}-\nabla_{X_{2}}X_{1})=\lambda_{1}\langle\nabla_{X_{1}}X_{2},X_{1}\rangle X_{1}-\lambda_{2}\langle\nabla_{X_{2}}X_{1},X_{2}\rangle X_{2}.
$$
From equation  (\ref{eq4.1}),
$$
(\lambda_{2}-\lambda_{1})\langle\nabla_{X_{1}}X_{2},X_{1}\rangle X_{1}+(\lambda_{2}-\lambda_{1})\langle\nabla_{X_{2}}X_{1},X_{2}\rangle X_{2}=c\nu (b_{2}X_{1}-b_{1}X_{2}),
$$
which implies, as $X_{1}$ and  $X_{2}$ are linearly independent fields, that
$$(\lambda_{2}-\lambda_{1})\langle\nabla_{X_{1}}X_{2},X_{1}\rangle = c\nu b_{2}\;\;\; \textrm{ and } \;\;\; (\lambda_{2}-\lambda_{1})\langle\nabla_{X_{2}}X_{1},X_{2}\rangle=-c\nu b_{1},$$ that is,
\begin{equation}\label{eq4.2}\begin{array}{ccc}
           \nabla_{X_{1}}X_{2} & = &\displaystyle\frac{c\nu b_{2}}{(\lambda_{2}-\lambda_{1})}X_{1},   \\
           \nabla_{X_{2}}X_{1} & = &\displaystyle\frac{-c\nu b_{1}}{(\lambda_{2}-\lambda_{1})}X_{2} .
         \end{array}
\end{equation}
Thus
\begin{equation}\label{eq4.21}\begin{array}{ccc}
           \nabla_{X_{1}}X_{1} & = &\displaystyle\frac{-c\nu b_{2}}{(\lambda_{2}-\lambda_{1})}X_{2},   \\
           \nabla_{X_{2}}X_{2} & = &\displaystyle\frac{c\nu b_{1}}{(\lambda_{2}-\lambda_{1})}X_{1} .
         \end{array}
\end{equation}
By (\ref{eq03}), we get
$$\nabla_{X_{1}}T=\nabla_{X_{1}}(b_{1}X_{1}+b_{2}X_{2})=X_{1}(b_{1})X_{1}+b_{1}\nabla_{X_{1}}X_{1}+X_{1}(b_{2})X_{2}+b_{2}\nabla_{X_{1}}X_{2}=\nu\lambda_{1}X_{1}\;\;\; \textrm{ and }$$ $$\nabla_{X_{2}}T=\nabla_{X_{2}}(b_{1}X_{1}+b_{2}X_{2})=X_{2}(b_{1})X_{1}+b_{1}\nabla_{X_{2}}X_{1}+X_{2}(b_{2})X_{2}+b_{2}\nabla_{X_{2}}X_{2}=\nu\lambda_{2}X_{2}.$$
Making the inner product of both equalities above with   $X_{1}$ and  $X_{2}$, we conclude
\begin{equation}\label{eq4.3}\begin{array}{ccccc}
                               X_{1}(b_{2}) & = & b_{1}\langle \nabla_{X_{1}}X_{2},X_{1}\rangle & = & \displaystyle\frac{c\nu b_{1}b_{2}}{\lambda_{2}-\lambda_{1}}, \\ \\
                               X_{2}(b_{1}) & = & b_{2}\langle \nabla_{X_{2}}X_{1},X_{2}\rangle & = & \displaystyle\frac{-c\nu b_{1}b_{2}}{\lambda_{2}-\lambda_{1}}.
                             \end{array}
\end{equation}

Therefore $X_{1}(b_{2})=-X_{2}(b_{1})$. Moreover,

\begin{equation}\label{eq4.31}\begin{array}{ccccc}
X_{1}(b_{1})+b_{2}\langle \nabla_{X_{1}}X_{2},X_{1}\rangle= \nu\lambda_{1}, \\ \\
 X_{2}(b_{2})+ b_{1}\langle \nabla_{X_{2}}X_{1},X_{2}\rangle =\nu\lambda_{2}.
\end{array}
\end{equation}
From equations (\ref{eq4.3}) and (\ref{eq4.31} ) we obtain
\begin{equation}\label{eq4.4}\begin{array}{ccc}
                                X_{1}(b_{1}) & = & \nu\lambda_{1} -\displaystyle\frac{c\nu (b_{2})^{2}}{\lambda_{2}-\lambda_{1}},\\ \\
                                X_{2}(b_{2}) & = & \nu\lambda_{2}+\displaystyle\frac{c\nu (b_{1})^{2}}{\lambda_{2}-\lambda_{1}}.
                             \end{array}
\end{equation}
So, by (\ref{eq04}),
\begin{equation}\label{eq4.5}\begin{array}{ccc}
                  X_{1}(\nu) & = & -\lambda_{1}b_{1}, \\
                  X_{2}(\nu) & = & -\lambda_{2}b_{2}.
                \end{array}
\end{equation}
Using now  Gauss equation we obtain
\begin{equation}\label{eq4.6} \langle R(X_{1},X_{2})X_{2},X_{1}\rangle=\lambda_{1}\lambda_{2}+c\nu^{2}.
\end{equation}

Observe that $$\begin{array}{rcl}
                \nabla_{[X_{1},X_{2}]}X_{2} & = & \nabla_{(\nabla_{X_{1}}X_{2}-\nabla_{X_{2}}X_{1})}X_{2} \\
                 & = & \nabla_{(\langle\nabla_{X_{1}}X_{2},X_{1}\rangle X_{1}-\langle\nabla_{X_{2}}X_{1}, X_{2}\rangle X_{2})}X_{2}\\
                 & = & \langle\nabla_{X_{1}}X_{2},X_{1}\rangle\nabla_{X_{1}}X_{2}-\langle\nabla_{X_{2}}X_{1}, X_{2}\rangle\nabla_{X_{2}}X_{2}.
              \end{array}$$
Then  $$\langle\nabla_{[X_{1},X_{2}]}X_{2},X_{1}\rangle=\langle\nabla_{X_{1}}X_{2},X_{1}\rangle^{2}+\langle\nabla_{X_{2}}X_{1}, X_{2}\rangle^{2}.$$
From equations  (\ref{eq4.6}) and  (\ref{eq4.2}), we obtain
\begin{equation}\label{eq4.7} \langle\nabla_{X_{1}}\nabla_{X_{2}}X_{2}-\nabla_{X_{2}}\nabla_{X_{1}}X_{2},X_{1}\rangle= \lambda_{1}\lambda_{2}+c\nu^{2}+\displaystyle\frac{\nu^{2}(b_{1}^{2}+b_{2}^{2})}{(\lambda_{2}-\lambda_{1})^{2}}.
\end{equation}

Observe, by equations  (\ref{eq4.2}) that  $\langle\nabla_{X_{1}}X_{2},\nabla_{X_{2}}X_{1}\rangle=0$ and so  $X_{2}\langle \nabla_{X_{1}}X_{2},X_{1}\rangle=\langle \nabla_{X_{2}}\nabla_{X_{1}}X_{2},X_{1}\rangle$. In order to compute  $\langle \nabla_{X_{2}}\nabla_{X_{1}}X_{2},X_{1}\rangle$ we derive the first equality of  (\ref{eq4.2}) with respect to  $X_{2}$ and use also the second equalities of  equations  (\ref{eq4.4}) and (\ref{eq4.5}) getting

\begin{equation}\label{eq4.8}-\langle\nabla_{X_{2}}\nabla_{X_{1}}X_{2},X_{1}\rangle=\displaystyle\frac{-c\lambda_{2}(\nu^{2}-b_{2}^{2})}{\lambda_{2}-\lambda_{1}}+
\displaystyle\frac{-\nu^{2}b_{1}^{2}}{(\lambda_{2}-\lambda_{1})^{2}}.
\end{equation}

From  equations  (\ref{eq4.3}) it follows that  $\langle\nabla_{X_{2}}X_{2},\nabla_{X_{1}}X_{1}\rangle=0$ and so  $X_{1}\langle \nabla_{X_{2}}X_{2},X_{1}\rangle=\langle \nabla_{X_{1}}\nabla_{X_{2}}X_{2},X_{1}\rangle$.
In order to compute  $\langle \nabla_{X_{1}}\nabla_{X_{2}}X_{2},X_{1}\rangle$ we derive the second equality  of  (\ref{eq4.3}) with respect to  $X_{1}$ and  use the first equation of   (\ref{eq4.4}) and  (\ref{eq4.5}) obtaining

\begin{equation}\label{eq4.9}\langle\nabla_{X_{1}}\nabla_{X_{2}}X_{2},X_{1}\rangle=\displaystyle\frac{c\lambda_{1}(\nu^{2}-b_{1}^{2})}{\lambda_{2}-\lambda_{1}}-
\displaystyle\frac{\nu^{2}b_{2}^{2}}{(\lambda_{2}-\lambda_{1})^{2}}.
\end{equation}

By summing equations  (\ref{eq4.9}) and  (\ref{eq4.8}), we get
\begin{equation}\label{eq4.10}
    \langle\nabla_{X_{1}}\nabla_{X_{2}}X_{2}-\nabla_{X_{2}}\nabla_{X_{1}}X_{2},X_{1}\rangle =-c\nu^{2}+\displaystyle\frac{c(\lambda_{2}b_{2}^{2}-\lambda_{1}b_{1}^{2})}{\lambda_{2}-\lambda_{1}}
    -\displaystyle\frac{\nu^{2}(b_{1}^{2}+b_{2}^{2})}{(\lambda_{2}-\lambda_{1})^{2}}.
\end{equation}
Finally from equations   (\ref{eq4.10}) and  (\ref{eq4.7}) we conclude that
$$ \lambda_{1}\lambda_{2}+2c\nu^{2}+\displaystyle\frac{c(\lambda_{1}b_{1}^{2}-\lambda_{2}b_{2}^{2})}{\lambda_{2}-\lambda_{1}}
+\displaystyle\frac{2\nu^{2}(b_{1}^{2}+b_{2}^{2})}{(\lambda_{2}-\lambda_{1})^{2}}=0.
$$
\end{proof}

Next result shows that a minimal surface of  $\mathbb{Q}^{2}_{c}\times\mathbb{R}$ with principal constant curvatures  is totally geodesic.

\begin{corollary}\label{min} The minimal surfaces of  $\mathbb{Q}^{2}_{c}\times\mathbb{R}$ with principal constant curvatures are totally geodesic.
\end{corollary}
\begin{proof}
Suppose that  there exist a minimal surface with two distinct constant principal curvatures   $\lambda_{2}=-\lambda_{1}$. From Proposition  \ref{prop21} we get
$$ -\lambda_{1}^{2}+2c\nu^{2}-\displaystyle\frac{c\lambda_{1}(b_{1}^{2}+b_{2}^{2})}{2\lambda_{1}}
+\displaystyle\frac{2\nu^{2}(b_{1}^{2}+b_{2}^{2})}{4\lambda_{1}^{2}}=0.
$$
We already know that  $\nu^{2}+b_{1}^{2}+b_{2}^{2}=1$ and thus
$$ -\lambda_{1}^{2}+2c\nu^{2}+\displaystyle\frac{c(\nu^{2}-1)}{2}
+\displaystyle\frac{2\nu^{2}(1-\nu^{2})}{4\lambda_{1}^{2}}=0,
$$ that is,
\begin{equation}\label{eq2.22}
\nu^{4}-\nu^{2}(1+5c\lambda_{1}^{2}) + \lambda_{1}^{2}(2\lambda_{1}^{2}+c)=0.
\end{equation}

So we obtain a biquadratic  equation  on the variable $\nu$, with constant real coefficients. If  equation  (\ref{eq2.22}) has a solution then the function  $\nu$ is  constant and  consequently $0~=~X_{i}(\nu)=-\langle AX_{i},T\rangle=-b_{i}\lambda_{i}$, with  $i\in\{1,2\}$. If  $\lambda_{1}=0$ then  $\lambda_{2}=-\lambda_{1}=0$, but this cannot occur since we are assuming  $\lambda_{1}\neq\lambda_{2}$. Then  $b_{1}=b_{2}=0$ and  $T=0$. Thus there does not exist a minimal surface  in  $\mathbb{Q}^{2}_{c}\times\mathbb{R}$ with two distinct constant principal curvatures.
\end{proof}

Proposition below shows that  a hypersurface in  $\QR$, $n\geq4$, with $\nu\neq0$, that has three constant principal curvatures of constant multiplicities   may not have two principal curvatures of  multiplicity one.

\begin{proposition} \label{dcpm1}
Let $\f$, $n\geq4$, be a hypersurface with  three constant distinct principal  curvatures  $\lambda$, $\mu$ and  $\gamma$ of constant multiplicities and suppose that  $\nu(p)~\neq~0$, for all  $p\in M^{n}$. Then there do not exist  two principal curvatures  of multiplicity one.
\end{proposition}

\begin{proof} Suppose there exist  two principal curvatures  $\lambda$ and  $\mu$ of multiplicity one. Let  $\{X_{1},X_{2},\ldots,X_{n}\}$ be a frame of principal orthonormal directions such that $AX_{1}~=~\lambda X_{1}$, $AX_{2}=\mu X_{2}$ and  $AX_{j}=\gamma X_{j}$, for $j\geq3$. From Remark  \ref{obs31}, we obtain  $T=b_{1}X_{1}+b_{2}X_{2}$, where  $b_{1},b_{2}\colon U\rightarrow\mathbb{R}$, $U\subset M^{n}$ are differentiable functions.

From  Codazzi equations, given in  (\ref{eq02}) we get
\begin{eqnarray}\label{eq3.1.0}
\nabla_{X_{1}}AX_{2}- \nabla_{X_{2}}AX_{1}-A[X_{1},X_{2}] & = & c\nu (b_{2}X_{1}-b_{1}X_{2}),\\
\label{eq3.11}
\nabla_{X_{1}}AX_{j}- \nabla_{X_{j}}AX_{1}-A[X_{1},X_{j}] & = & -c\nu b_{1}X_{j}, \mbox{ for each } j\in\{3,\ldots,n\},\\
\label{eq3.12}
\nabla_{X_{2}}AX_{j}- \nabla_{X_{j}}AX_{2}-A[X_{2},X_{j}] & = & -c\nu b_{2}X_{j}, \mbox{ for each } j\in\{3,\ldots,n\},\\
\label{eq3.13}
\nabla_{X_{\beta}}AX_{j}- \nabla_{X_{j}}AX_{\beta}-A[X_{\beta},X_{j}] & = &  0, \mbox{ for } j\in\{3,\ldots,n\} \textrm{ and } \beta\neq1,2,j.
\end{eqnarray}
From equation (\ref{eq3.1.0}), we obtain
$$\mu\nabla_{X_{1}}X_{2}- \lambda\nabla_{X_{2}}X_{1}-A(\nabla_{X_{1}}X_{2}-\nabla_{X_{2}}X_{1}) = c\nu (b_{2}X_{1}-b_{1}X_{2}),$$
that is,
$$\sum_{k=1}^{n}\langle \nabla_{X_{1}}X_{2},X_{k}\rangle(\mu I-A)X_{k}+\sum_{l=1}^{n}\langle \nabla_{X_{2}}X_{1},X_{l}\rangle(A-\lambda I)X_{l}-c\nu b_{2}X_{1}+c\nu b_{1}X_{2}=0.$$
Thus,
\begin{equation}\label{eq3.14}
\begin{array}{rcl}
\langle\nabla_{X_{1}}X_{2},X_{1}\rangle(\mu-\lambda)-c\nu b_{2} & = & 0,\\
\langle\nabla_{X_{2}}X_{1},X_{2}\rangle(\mu-\lambda)+c\nu b_{1} & = & 0
\end{array}
\end{equation}
\begin{equation}\label{eq3.15}
\textrm{ and \;\; } \langle\nabla_{X_{1}}X_{2},X_{j}\rangle(\mu-\gamma)+\langle\nabla_{X_{2}}X_{1},X_{j}\rangle(\gamma-\lambda)=0, \textrm{ for each } j\in\{3,\ldots,n\}.
\end{equation}
Proceeding analogously with equations   (\ref{eq3.11}), (\ref{eq3.12}) and  (\ref{eq3.13}) from equation  (\ref{eq3.11}) we obtain for each  $j\in\{3,\ldots,n\}$,
\begin{eqnarray}\label{eq3.16}
\langle\nabla_{X_{1}}X_{j},X_{1}\rangle(\gamma-\lambda) & = & 0, \\
\label{eq3.17}
\langle\nabla_{X_{j}}X_{1},X_{j}\rangle(\gamma-\lambda) +c\nu b_{1} & = & 0, \\
\label{eq3.18}
\langle\nabla_{X_{j}}X_{1},X_{\beta}\rangle(\gamma-\lambda) & = & 0, \;\; \beta\neq 1,2,j, \textrm{ \;and }\\
\label{eq3.19}
\langle\nabla_{X_{1}}X_{j},X_{2}\rangle(\gamma-\mu)+\langle\nabla_{X_{j}}X_{1},X_{2}\rangle(\mu-\lambda) & = & 0.
\end{eqnarray}

By equation  (\ref{eq3.12}) we conclude for each $j\in\{3,\ldots,n\}$ that
\begin{eqnarray}\label{eq3.20}
\langle\nabla_{X_{2}}X_{j},X_{2}\rangle(\gamma-\mu) & = & 0, \\
\label{eq3.21}
\langle\nabla_{X_{j}}X_{2},X_{j}\rangle(\gamma-\mu) +c\nu b_{2} & = & 0, \\
\label{eq3.22}
\langle\nabla_{X_{j}}X_{2},X_{\beta}\rangle(\gamma-\mu) & = & 0, \;\; \beta\neq 1,2,j, \textrm{ \;and }\\
\label{eq3.23}
\langle\nabla_{X_{2}}X_{j},X_{1}\rangle(\gamma-\lambda)+\langle\nabla_{X_{j}}X_{2},X_{1}\rangle(\lambda-\mu) & = & 0.
\end{eqnarray}

By using equation  (\ref{eq3.13}) we get for each  $j\in\{3,\ldots,n\}$ and $\beta\neq1,2,j$,
\begin{equation}\label{eq3.24}
\begin{array}{rcl}
\langle\nabla_{X_{\beta}}X_{j},X_{1}\rangle- \langle\nabla_{X_{j}}X_{\beta},X_{1}\rangle & = & 0,\\
\langle\nabla_{X_{\beta}}X_{j},X_{2}\rangle-\langle\nabla_{X_{j}}X_{\beta},X_{2}\rangle & = & 0.
\end{array}
\end{equation}
From  equations (\ref{eq3.18}), (\ref{eq3.22}) and  (\ref{eq3.24}), it follows that
\begin{equation}\label{eq3.25}
\langle\nabla_{X_{\beta}}X_{j},X_{1}\rangle=0 \;\;\;\textrm{ and }\;\;\;\langle\nabla_{X_{\beta}}X_{j},X_{2}\rangle=0.
\end{equation}
from  equations  (\ref{eq3.14}) to (\ref{eq3.23}) and  (\ref{eq3.25}) we conclude, for each $j\in\{3,\ldots,n\}$, that
\begin{eqnarray}\label{eq3.26}
\nabla_{X_{1}}X_{1} & = & -\displaystyle\frac{c\nu b_{2}}{\mu-\lambda}X_{2}, \\
\label{eq3.27}
\nabla_{X_{1}}X_{2} & = & \displaystyle\frac{c\nu b_{2}}{\mu-\lambda}X_{1}+\sum_{j=3}^{n}\langle\nabla_{X_{2}}X_{1},X_{j}\rangle\displaystyle\frac{(\lambda-\gamma)}{\mu-\gamma} X_{j}, \\
\label{eq3.28}
\nabla_{X_{1}}X_{j} & = & \langle\nabla_{X_{j}}X_{1},X_{2}\rangle\displaystyle\frac{(\lambda-\mu)}{\gamma-\mu} X_{2}+\sum_{\beta\neq1,2,j}\langle\nabla_{X_{1}}X_{j},X_{\beta}\rangle X_{\beta}, \\
\label{eq3.29}
\nabla_{X_{2}}X_{1} & = & -\displaystyle\frac{c\nu b_{1}}{\mu-\lambda}X_{2}+\sum_{j=3}^{n}\langle\nabla_{X_{2}}X_{1},X_{j}\rangle X_{j}, \\
\label{eq3.30}
\nabla_{X_{2}}X_{2} & = & \displaystyle\frac{c\nu b_{1}}{\mu-\lambda}X_{1}, \\
\label{eq3.31}
\nabla_{X_{2}}X_{j} & = & \langle\nabla_{X_{j}}X_{2},X_{1}\rangle\displaystyle\frac{(\mu-\lambda)}{\gamma-\lambda}X_{1}+\sum_{\beta\neq1,2,j}\langle\nabla_{X_{2}}X_{j},X_{\beta}\rangle X_{\beta}, \\
\label{eq3.32}
\nabla_{X_{j}}X_{1} & = & \langle \nabla_{X_{j}}X_{1},X_{2}\rangle X_{2}-\displaystyle\frac{c\nu b_{1}}{\gamma-\lambda}X_{j}, \\
\label{eq3.33}
\nabla_{X_{j}}X_{2} & = & \langle \nabla_{X_{j}}X_{2},X_{1}\rangle X_{1}-\displaystyle\frac{c\nu b_{2}}{\gamma-\mu}X_{j},\\
\label{eq3.34}
\nabla_{X_{j}}X_{j} & = & \displaystyle\frac{c\nu b_{1}}{\gamma-\lambda}X_{1}+\displaystyle\frac{c\nu b_{2}}{\gamma-\mu}X_{2}+\sum_{\beta\neq1,2,j}\langle\nabla_{X_{j}}X_{j},X_{\beta}\rangle X_{\beta}, \\
\label{eq3.35}
\nabla_{X_{\beta}}X_{j} & = & \langle\nabla_{X_{\beta}}X_{j},X_{\beta}\rangle X_{\beta}+\displaystyle\sum_{l\neq1,2,j,\beta}\langle\nabla_{X_{\beta}}X_{j},X_{l}\rangle X_{l}.
\end{eqnarray}

From  equalities  (\ref{eq03}) and  (\ref{eq04}), for all  $X\in TM^{n}$, we get
\begin{equation}\label{eq3.40}
\begin{array}{rcl}
  X_{1}(\nu) & = & -\lambda b_{1} \\
  X_{2}(\nu) & = & -\mu b_{2} \\
  X_{j}(\nu) & = & 0,\;\;j\in\{3,\ldots,n\}.
\end{array}
\end{equation}
Moreover from
$$\nabla _{X_{1}}(b_{1}X_{1}+b_{2}X_{2})=X_{1}(b_{1})X_{1}+b_{1}\nabla _{X_{1}}X_{1}+X_{1}(b_{2})X_{2}+b_{2}\nabla _{X_{1}}X_{2}=\nu \lambda X_{1},$$
we obtain
\begin{eqnarray}
\label{eq3.41}
X_{1}(b_{1}) & = & \nu\lambda - \displaystyle\frac{c\nu b_{2}^{2}}{\mu-\lambda},\\
\label{eq3.42}
X_{1}(b_{2}) & = & \displaystyle\frac{c\nu b_{1}b_{2}}{\mu-\lambda},\\
\label{eq3.43}
b_{2}\langle\nabla _{X_{1}}X_{2},X_{j}\rangle & = & 0, \;\;\textrm{ for each } j\in\{3,\ldots,n\}.
\end{eqnarray}

Analogously, deriving  $T$ with respect to  $X_{2}$, we obtain
$$\nabla _{X_{2}}(b_{1}X_{1}+b_{2}X_{2})=X_{2}(b_{1})X_{1}+b_{1}\nabla _{X_{2}}X_{1}+X_{2}(b_{2})X_{2}+b_{2}\nabla _{X_{2}}X_{2}=\nu \mu X_{2},$$
and therefore,
\begin{eqnarray}
\label{eq3.44}
X_{2}(b_{1}) & = & -\displaystyle\frac{c\nu b_{1}b_{2}}{\mu-\lambda},\\
\label{eq3.45}
X_{2}(b_{2}) & = & \nu\mu + \displaystyle\frac{c\nu b_{1}^{2}}{\mu-\lambda},\\
\label{eq3.46}
b_{1}\langle\nabla _{X_{2}}X_{1},X_{j}\rangle & = & 0, \;\;\textrm{ for each } j\in\{3,\ldots,n\}.
\end{eqnarray}
On one hand using Gauss equation (\ref{eq01}), we obtain
\begin{eqnarray}
\label{eq3.47}
K(X_{1},X_{2}) & = & \lambda\mu+c\nu^{2},\\
\label{eq3.48}
K(X_{1},X_{j}) & = & \lambda\gamma+c(1-b_{1}^{2}),\;\;j\in\{3,\ldots,n\},\\
\label{eq3.49}
K(X_{2},X_{j}) & = & \mu\gamma+c(1-b_{2}^{2}),\;\;j\in\{3,\ldots,n\}.
\end{eqnarray}
On the other hand, we know that  $$K(X_{1},X_{2})=\langle R(X_{1},X_{2})X_{2},X_{1}\rangle=\langle\nabla_{X_{1}}\nabla _{X_{2}}X_{2}-\nabla_{X_{2}}\nabla _{X_{1}}X_{2}-\nabla _{[X_{1},X_{2}]}X_{2},X_{1}\rangle.$$
From equations  (\ref{eq3.26}) and (\ref{eq3.30}) we get
$$\langle\nabla_{X_{1}}\nabla _{X_{2}}X_{2},X_{1}\rangle=X_{1}\left(\displaystyle\frac{c\nu b_{1}}{\mu-\lambda}\right),$$
and from  (\ref{eq3.40}) and  (\ref{eq3.41}),
\begin{equation}\label{eq3.50}
\langle\nabla_{X_{1}}\nabla _{X_{2}}X_{2},X_{1}\rangle=\displaystyle\frac{c\lambda(\nu^{2}-b_{1}^{2})}{\mu-\lambda}-\displaystyle\frac{\nu^{2}b_{2}^{2}}{(\mu-\lambda)^{2}}.
\end{equation}
By using now equations  (\ref{eq3.27}), (\ref{eq3.29}), (\ref{eq3.40}) and  (\ref{eq3.45}), we arrive to
$$
\begin{array}{rcl}
-\langle \nabla_{X_{2}}\nabla _{X_{1}}X_{2},X_{1} \rangle & = & \langle \nabla _{X_{1}}X_{2},\nabla_{X_{2}}X_{1} \rangle-X_{2}\langle\nabla _{X_{1}}X_{2},X_{1} \rangle  \\
  & = & \displaystyle\sum^{n}_{j=3}\langle\nabla _{X_{2}}X_{1},X_{j} \rangle^{2}\displaystyle\frac{(\lambda-\gamma)}{\mu-\gamma}- X_{2}\left(\displaystyle\frac{c\nu b_{2}}{\mu-\lambda}\right)
\end{array}
$$
and so
\begin{equation}\label{eq3.51}
-\langle \nabla_{X_{2}}\nabla _{X_{1}}X_{2},X_{1} \rangle=\displaystyle\sum^{n}_{j=3}\langle\nabla _{X_{2}}X_{1},X_{j} \rangle^{2}\displaystyle\frac{(\lambda-\gamma)}{\mu-\gamma}- \displaystyle\frac{c\mu(\nu^{2}-b_{2}^{2})}{\mu-\lambda}-\displaystyle\frac{\nu^{2}b_{1}^{2}}{(\mu-\lambda)^{2}}.
\end{equation}

Observe that

$$
 \begin{array}{rcl}
 \nabla_{[X_{1},X_{2}]}X_{2} & = & \displaystyle\frac{c\nu b_{2}}{\mu-\lambda}\nabla_{X_{1}}X_{2}+\sum_{j=3}^{n}\langle\nabla_{X_{2}}X_{1},X_{j}\rangle
 \displaystyle\frac{(\lambda-\gamma)}{\mu-\gamma}\nabla_{X_{j}}X_{2}\\
   &  &  +\displaystyle\frac{c\nu b_{1}}{\mu-\lambda}\nabla_{X_{2}}X_{2}-\sum_{j=3}^{n}\langle\nabla_{X_{2}}X_{1},X_{j}\rangle\nabla_{X_{j}}X_{2}.
\end{array}
$$
Thus,
$$-\langle\nabla_{[X_{1},X_{2}]}X_{2},X_{1}\rangle=-\displaystyle\frac{\nu^{2}(b_{1}^{2}+b_{2}^{2})}{(\mu-\lambda)^{2}}+\sum_{j=3}^{n}\langle\nabla_{X_{2}}X_{1},X_{j}\rangle
\langle\nabla_{X_{j}}X_{2},X_{1}\rangle\left(1-\displaystyle\frac{(\lambda-\gamma)}{\mu-\gamma}\right).$$
From (\ref{eq3.31}) and  (\ref{eq3.33}), we get
$$\langle\nabla_{X_{j}}X_{2},X_{1}\rangle=\langle\nabla_{X_{2}}X_{j},X_{1}\rangle\displaystyle\frac{(\gamma-\lambda)}{\mu-\lambda}$$ and so
\begin{equation}\label{eq3.52}
-\langle\nabla_{[X_{1},X_{2}]}X_{2},X_{1}\rangle=-\displaystyle\frac{\nu^{2}(b_{1}^{2}+b_{2}^{2})}{(\mu-\lambda)^{2}}-\sum_{j=3}^{n}
\langle\nabla_{X_{2}}X_{j},X_{1}\rangle^{2}\displaystyle\frac{(\gamma-\lambda)}{\mu-\lambda}\left(1-\displaystyle\frac{(\lambda-\gamma)}{\mu-\gamma}\right).
\end{equation}
By summing  (\ref{eq3.50}) with  (\ref{eq3.51}) and  (\ref{eq3.52}) and comparing with  (\ref{eq3.47}), we get
\begin{equation}\label{eq3.5}
\displaystyle\frac{2\nu^{2}(1-\nu^{2})}{(\mu-\lambda)^{2}}+2c\nu^{2}+\lambda\mu+\displaystyle\frac{c(b_{1}^{2}\lambda-b_{2}^{2}\mu)}{\mu-\lambda}-
\displaystyle\frac{2(\lambda-\gamma)}{\mu-\gamma}\sum_{j=3}^{n}\langle\nabla_{X_{2}}X_{1},X_{j}\rangle^{2}=0.
\end{equation}
Proceeding analogously  for  $K(X_{1},X_{j})$ and  $K(X_{2},X_{j})$ we conclude for each  $j\geq3$, that

\begin{eqnarray}\nonumber
& & \displaystyle\frac{c\lambda(\nu^{2}-b_{1}^{2})}{\gamma-\lambda}-\displaystyle\frac{\nu^{2}}{\gamma-\lambda}\left(\displaystyle\frac{b_{1}^{2}}{\gamma-\lambda}+
\displaystyle\frac{b_{2}^{2}}{\mu-\lambda}\right)+\displaystyle\frac{\nu^{2}b_{2}^{2}}{(\mu-\lambda)(\gamma-\mu)}-\lambda\gamma-c(1-b_{1}^{2})   \\ \label{eq3.6}
&  & -\displaystyle\frac{2(\mu-\lambda)}{\gamma-\mu}\langle\nabla_{X_{j}}X_{1},X_{2}\rangle^{2}=0,
\end{eqnarray}

\begin{eqnarray}\nonumber
&  & \displaystyle\frac{c\mu(\nu^{2}-b_{2}^{2})}{\gamma-\mu}+\displaystyle\frac{\nu^{2}}{\gamma-\mu}\left(\displaystyle\frac{b_{1}^{2}}{\mu-\lambda}-
\displaystyle\frac{b_{2}^{2}}{\gamma-\mu}\right)-\displaystyle\frac{\nu^{2}b_{1}^{2}}{(\gamma-\lambda)(\mu-\lambda)}-\mu\gamma-c(1-b_{2}^{2})
   \\ \label{eq3.7}
&  & + \displaystyle\frac{2(\mu-\lambda)}{\gamma-\lambda}\langle\nabla_{X_{j}}X_{1},X_{2}\rangle^{2}=0.
\end{eqnarray}

From  (\ref{eq3.46}) we get  $b_{1}\langle\nabla _{X_{2}}X_{1},X_{j}\rangle =0$, for each  $j\in\{3,\ldots,n\}$.
Suppose that $b_{1}(p)=0$, for all  $p\in U\subset M^{n}$. Then
$T$ is a principal direction. As we are supposing $\nu\neq0$ from  Theorem \ref{ctdp} we get  $c=-1$ and  $g=2$, which is against the hypothesis  $g=3$.
So it must exist a  $p_{0}$ such that  $b_{1}(p_{0})\neq 0$. Since the function  $b_{1}$ is continuous  there  exist a neighborhood $V\subset U\subset M^{n}$ of $p_{0}$ such that  $b_{1}(p)\neq 0$ for all $p\in V$. Thus  $\langle\nabla _{X_{2}}X_{1},X_{j}\rangle =0 $  in  $V$, for each  $j\in\{3,\ldots,n\}$.

From equations  (\ref{eq3.15}) and (\ref{eq3.19}) we conclude that
$$\langle\nabla _{X_{2}}X_{1},X_{j}\rangle =0 \Leftrightarrow \langle\nabla _{X_{1}}X_{2},X_{j}\rangle =0 \Leftrightarrow\langle\nabla _{X_{j}}X_{1},X_{2}\rangle =0,\;\;\textrm{ for each  } j\in\{3,\ldots,n\}.$$
Therefore, from equations  (\ref{eq3.5}), (\ref{eq3.6}) and  (\ref{eq3.7}) we obtain in  $V$, respectively,
\begin{equation}\label{eq3.8}
\displaystyle\frac{2\nu^{2}(1-\nu^{2})}{(\mu-\lambda)^{2}}+2c\nu^{2}+\lambda\mu+\displaystyle\frac{c(b_{1}^{2}\lambda-b_{2}^{2}\mu)}{\mu-\lambda}=0,
\end{equation}
\begin{equation}\label{eq3.9}
\displaystyle\frac{c\lambda(\nu^{2}-b_{1}^{2})}{\gamma-\lambda}-\displaystyle\frac{\nu^{2}b_{1}^{2}}{(\gamma-\lambda)^{2}}+
\displaystyle\frac{\nu^{2}b_{2}^{2}}{(\gamma-\mu)(\gamma-\lambda)}-\lambda\gamma-c(1-b_{1}^{2})=0,
\end{equation}
\begin{equation}\label{eq3.10}
\displaystyle\frac{c\mu(\nu^{2}-b_{2}^{2})}{\gamma-\mu}-\displaystyle\frac{\nu^{2}b_{2}^{2}}{(\gamma-\mu)^{2}}+
\displaystyle\frac{\nu^{2}b_{1}^{2}}{(\gamma-\mu)(\gamma-\lambda)}-\mu\gamma-c(1-b_{2}^{2})=0.
\end{equation}
Replacing  $b_{2}^{2}$ by $1-\nu^{2}-b^{2}_{1}$ in equation (\ref{eq3.8}), we get
$$-cb_{1}^{2}\displaystyle\frac{(\lambda+\mu)}{\mu-\lambda}=
 \lambda\mu+2c\nu^{2}-\displaystyle\frac{c\mu(1-\nu^{2})}{\mu-\lambda}
+\displaystyle\frac{2\nu^{2}(1-\nu^{2})}{(\mu-\lambda)^{2}}.
$$

With analogous arguments used in  Proposition  \ref{min} we conclude that  $\lambda+\mu\neq0$. Then
\begin{equation}\label{eq3.4.0}
b_{1}^{2}=\nu^{4}\displaystyle\frac{2c}{(\mu+\lambda)(\mu-\lambda)}+\nu^{2}\displaystyle\frac{(\mu-\lambda)(-3\mu+2\lambda)-2c}{(\mu+\lambda)(\mu-\lambda)}+
\displaystyle\frac{\mu}{\mu+\lambda}\left(1-c\lambda(\mu-\lambda)\right).
\end{equation}
By summing  (\ref{eq3.9}) with  (\ref{eq3.10}) we get

\begin{eqnarray}\nonumber
&  & c\nu^{2}\left(\displaystyle\frac{\lambda}{\gamma-\lambda}+\displaystyle\frac{\mu}{\gamma-\mu}\right)-\displaystyle\frac{c\lambda b_{1}^{2}}{\gamma-\lambda}-\displaystyle\frac{c\mu b_{2}^{2}}{\gamma-\mu}-
\displaystyle\frac{\nu^{2}b_{1}^{2}}{(\gamma-\lambda)^{2}}-\displaystyle\frac{\nu^{2}b_{2}^{2}}{(\gamma-\mu)^{2}}\\ \label{eq3.4.1}
&  & +\displaystyle\frac{\nu^{2}(1-\nu^{2})}{(\gamma-\mu)(\gamma-\lambda)} -\gamma(\mu+\lambda)-c(1+\nu^{2})=0.
\end{eqnarray}

Now replacing  $b_{2}^{2}$ by  $1-\nu^{2}-b_{1}^{2}$  in  (\ref{eq3.4.1}) we get

\begin{eqnarray} \nonumber
& &\nu^{4}\displaystyle\frac{(\mu-\lambda)}{(\gamma-\mu)^{2}(\gamma-\lambda)}+\nu^{2}\left(\displaystyle\frac{-(\mu-\lambda)+c\lambda(\gamma-\mu)^{2}+2c\mu(\gamma-\mu)(\gamma-\lambda)}{(\gamma-\mu)^{2}(\gamma-\lambda)} -c\right)-\displaystyle\frac{c\gamma}{\gamma-\mu} \\  \label{eq3.4.2}
&  & -\gamma(\mu+\lambda)+cb_{1}^{2}\left(\displaystyle\frac{\mu}{\gamma-\mu}-\displaystyle\frac{\lambda}{\gamma-\lambda}\right)+
\nu^{2}b_{1}^{2}\left(\displaystyle\frac{1}{(\gamma-\mu)^{2}}-\displaystyle\frac{1}{(\gamma-\lambda)^{2}}\right)=0.
\end{eqnarray}

After replacing  (\ref{eq3.4.0}) in (\ref{eq3.4.2}) we observe that  the greatest  power of  $\nu$ in the last equation is  $\nu^{6}$ whose coefficient is  $\left(\displaystyle\frac{1}{(\gamma-\mu)^{2}}-\displaystyle\frac{1}{(\gamma-\lambda)^{2}}\right)\displaystyle\frac{2c}{(\mu+\lambda)(\mu-\lambda)}$. If that  term is different from zero the equation has grade six and has constant real coefficients in the variable $\nu$. If not, i.e. if  $(\gamma-\lambda)^{2}=(\gamma-\mu)^{2}$ we have  $\gamma-\lambda=\gamma-\mu$ or $\gamma-\lambda=-(\gamma-\mu)$. If $\gamma-\lambda=\gamma-\mu$ then  $\lambda=\mu$ which is against the hypothesis  $\lambda\neq \mu$. If $\gamma-\lambda=-(\gamma-\mu)$ then $\mu-\lambda=2(\gamma-\lambda)$. Finally, the greatest power of  $\nu$ would be of fourth order and the coefficient of  $\nu^{4}$ is $\displaystyle\frac{1}{(\gamma-\mu)^{2}}$ which does not vanish. In that case we would have  a biquadratic equation with real constant coefficients on the variable  $\nu$.

Therefore, assuming the existence of two constant principal curvatures  with multiplicity one and  $\nu(p)\neq0$ for all  $p\in M^{n}$ the function  $\nu$ must  satisfy equation (\ref{eq3.4.2}). In that case the function  $\nu$ would be constant and consequently  $T$ would be a principal direction. Now using Theorem \ref{ctdp} we would obtain  $c=-1$ and  $g=2$, against the hypothesis  $g=3$.

Finally the conclusion is that  there does not  exist two  principal constant curvatures with  multiplicity one if we suppose also  $\nu(p)\neq0$, for all $p\in M^{n}$.
\end{proof}

Taking into account the previous results we are ready to prove  the local  classification of  hypersurfaces in $\QR$, $n\neq3$, having constant principal curvatures and $g\in\{1,2,3\}$.

\begin{theorem} \label{teoclass}
Let  $\f$ be a hypersurface  with constant principal curvatures.
\begin{description}
  \item[(i)] If  $g=1$ and  $n\geq2$  then  $f(M^{n})$ is an open subset of  $\mathbb{Q}^{n}_{c}\times\{t_{0}\}$, for any $t_{0}\in\mathbb{R}$ or an open subset of a Riemannian product $M^{n-1}\times\mathbb{R}$. In the last case if  $c=1$, $M^{n-1}$ is a totally geodesic sphere in  $\mathbb{S}^{n}$ and if  $c=-1$, $M^{n-1}$ is a totally geodesic hyperplane in  $\mathbb{H}^{n}$.
  \item[(ii)] If  $g=2$ and  $n\geq2$ then  $c=-1$ and  $f$ is locally given by  $f(x,s)~=~ g_{s}(x)+Bs\dt$, for some  $B\in \mathbb{R}$, $B>0$, with  $M^{n}=M^{n-1}\times I$,  where  $g_{s}$ is a family of horospheres in  $\mathbb{H}^{n}$, or  $f(M^{n})$ is an open subset of a Riemannian product $M^{n-1}\times\mathbb{R}$. In the last case, if  $c=1$ then $M^{n-1}$ is a non totally geodesic sphere in $\mathbb{S}^{n}$ and if  $c=-1$, $M^{n-1}$ is an equidistant hipersurface, a horosphere or a  hypersphere in $\mathbb{H}^{n}$.

  \item[(iii)] Suppose that $g=3$, $n\geq4$ and the multiplicities of the principal curvatures are constant. Then if  $c=1$, $f(M^{n})$  is an open subset of a Riemannian product $\mathbb{S}^{p}(r)\times\mathbb{S}^{q}(s)\times\mathbb{R}$, with  $n=p+q+1$ and  $r^{2}+s^{2}=1$ or an open subset of the product  $M^{n-1}\times\mathbb{R}$, where $M^{n-1}$ is a Cartan's hypersurface with  $n\in\{4, 7, 13, 25\}$.  If  $c=-1$, $f(M^{n})$ is an open subset of the Riemannian product $\mathbb{S}^{k}\times\mathbb{H}^{n-k-1}\times\mathbb{R}$.

\end{description}
\end{theorem}

\begin{proof}

(i) Suppose that  $g=1$.  If  $\nu\equiv 0$ from Remark \ref{obsvctd} we infer that the field  $T$ is a  principal direction with corresponding principal curvature $\lambda=0$. Then $f$ an umbilical immersion  implies that all the principal curvatures are zero, i.e,  $f$ is  totally  geodesic. In this case, $f$ is an open subset of a Riemannian product $M^{n-1}\times\mathbb{R}$, where $M^{n-1}$ is a totally geodesic hypersurface in  $\mathbb{Q}_{c}^{n}$.

Let us now suppose that  $\nu(p) \neq 0$, for all  $p\in M$.  Take  $\{X_{1}, X_{2}, ..., X_{n}\}$ a local orthonormal frame field of the immersion  $f$. We can write  $T=\displaystyle\sum_{i=1}^{n}b_{i}X_{i}$.
From  hypothesis we have  $AX_{i}=\lambda X_{i}$ for all  $i\in \{1,...,n\}$ where  $\lambda$ is constant in  $\mathbb{R}$.

From Codazzi's equations we get,
$$\nabla_{X_{i}}AX_{j}- \nabla_{X_{j}}AX_{i}-A[X_{i},X_{j}]=c\nu (b_{j}X_{i}-b_{i}X_{j}),$$
 which implies $0=\lambda[X_{i},X_{j}]-\lambda[X_{i},X_{j}]=c\nu (b_{j}X_{i}-b_{i}X_{j})$. Since  $\nu\neq 0$ and the fields  $X_{i}$,  $X_{j}$ are  linearly  independent for $i\neq j$ it follows that  $b_{i}=b_{j}=0$ for all $i,j\in \{1,...,n\}$. Then  $T=0$ and  $f(M^{n})$ is locally an open subset  of $\mathbb{Q}^{n}_{c}\times \{t\}$, which proves   item (i).

(ii) Suppose now  $g=2$. By \cite[Proposition 2.2]{Ry1} the multiplicities of the principal curvatures  are constant. Let us consider two subcases $n=2$ and $n\geq3$.

First case: $n=2$ \\
Replace $b_{2}^{2}=1-\nu^{2}-b^{2}_{1}$  in equation (\ref{eq2.21}) to get,
$$-cb_{1}^{2}\displaystyle\frac{(\lambda_{1}+\lambda_{2})}{\lambda_{2}-\lambda_{1}}=
 \lambda_{1}\lambda_{2}+2c\nu^{2}-\displaystyle\frac{c\lambda_{2}(1-\nu^{2})}{\lambda_{2}-\lambda_{1}}
+\displaystyle\frac{2\nu^{2}(1-\nu^{2})}{(\lambda_{2}-\lambda_{1})^{2}}.
$$
From Corollary \ref{min} we obtain  $\lambda_{1}+\lambda_{2}\neq0$ and so,
\begin{equation}\label{eq23}
b_{1}^{2}=-c\displaystyle\frac{(\lambda_{2}-\lambda_{1})}{\lambda_{1}+\lambda_{2}}\left\{\lambda_{1}\lambda_{2}+2c\nu^{2}-\displaystyle\frac{c\lambda_{2}(1-\nu^{2})}{\lambda_{2}-\lambda_{1}}
+\displaystyle\frac{2\nu^{2}(1-\nu^{2})}{(\lambda_{2}-\lambda_{1})^{2}}\right\}.
\end{equation}
Replace also  $b_{1}^{2}=1-\nu^{2}-b^{2}_{2}$ on equation  (\ref{eq2.21}) to get,
\begin{equation}\label{eq24}
-2\nu^{4}+\nu^{2}D+b_{2}^{2}E=F,
\end{equation}
where $D=2c(\lambda_{2}-\lambda_{1})^{2}-c\lambda_{1}(\lambda_{2}-\lambda_{1})+2$, $E=-c(\lambda_{2}^{2}-\lambda_{1}^{2})$ and
$F=-\lambda_{1}\lambda_{2}(\lambda_{2}-\lambda_{1})^{2}-c\lambda_{1}(\lambda_{2}-\lambda_{1})$.
Deriving equation  (\ref{eq24}) with respect to  $X_{2}$ we get
\begin{equation}\label{eq25}
-8\nu^{3}X_{2}(\nu)+2\nu X_{2}(\nu)D+2b_{2}X_{2}(b_{2})E=0.
\end{equation}
If $\lambda_{2}=0$ from equation  (\ref{eq4.5}), $X_{2}(\nu)=0$ and from equation  (\ref{eq25}), $b_{2}X_{2}(b_{2})=0$. Then $b_{2}=0$ or the expression for  $X_{2}(b_{2})$ given by  (\ref{eq4.4}), $\nu=0$ or $b_{1}=0$.
If  $\nu\equiv0$, the field  $T$ is a principal direction. If  $\nu\neq0$ then  $b_{2}=0$ or $b_{1}=0$. Suppose that there exist $x_{0}$ such that  $b_{1}(x_{0})\neq0$. Then by continuity there exist a neighborhood  $V$ of $x_{0}$ such that  $b_{1}(x)\ne0$ for all $x\in V$. Then $b_{2}(x)=0$ for all $x\in V$ and  $T$ is a principal direction. Finally if  $\lambda_{2}=0$ it follows that  $T$ is a  principal direction.

Suppose that  $\lambda_{2}\neq0$. From  (\ref{eq4.5}), $b_{2}=-\displaystyle\frac{X_{2}(\nu)}{\lambda_{2}}$ and from equation  (\ref{eq25}) we get
$$X_{2}(\nu)\left\{-8\nu^{3}+2\nu D-\displaystyle\frac{2E}{\lambda_{2}}X_{2}(b_{2})\right\}=0.$$
Therefore $X_{2}(\nu)=0$ or  $-8\nu^{3}+2\nu D-\displaystyle\frac{2E}{\lambda_{2}}X_{2}(b_{2})=0$.

If $X_{2}(\nu)=0$ we obtain  $b_{2}=0$ and  $T$ is a  principal direction. If not,
$$-8\nu^{3}+2\nu D-\displaystyle\frac{2E}{\lambda_{2}}\left(\nu\lambda_{2}+\displaystyle\frac{c\nu (b_{1})^{2}}{\lambda_{2}-\lambda_{1}}\right)=0.$$
 So,
\begin{equation}\label{eq26}
-8\nu^{3}+4\nu+4c\nu(\lambda_{2}-\lambda_{1})^{2}-2c\nu\lambda_{1}(\lambda_{2}-\lambda_{1})+2c\nu(\lambda_{2}^{2}-\lambda_{1}^{2})+
\displaystyle\frac{2\nu(\lambda_{2}+\lambda_{1})}{\lambda_{2}}b_{1}^{2}=0.
\end{equation}
From  (\ref{eq23}) we get
\begin{equation}\label{eq27}
\displaystyle\frac{2\nu(\lambda_{2}+\lambda_{1})}{\lambda_{2}}b_{1}^{2}=-2c\nu\lambda_{1}(\lambda_{2}-\lambda_{1})+2\nu(1-\nu^{2})-
\displaystyle\frac{4\nu^{3}(\lambda_{2}-\lambda_{1})}{\lambda_{2}}-\displaystyle\frac{4c\nu^{3}(1-\nu^{2})}{\lambda_{2}(\lambda_{2}-\lambda_{1})}.
\end{equation}
By replacing  (\ref{eq27}) in  (\ref{eq26}) it follows that
$$\displaystyle\frac{4c}{\lambda_{2}(\lambda_{2}-\lambda_{1})}\nu^{5}-
\nu^{3}\left(10+4\displaystyle\frac{(\lambda_{2}-\lambda_{1})^{2}+c}{\lambda_{2}(\lambda_{2}-\lambda_{1})}\right)+
6\nu\left(1+c(\lambda_{2}-\lambda_{1})^{2}\right)=0.$$
Observe that $\displaystyle\frac{4c}{\lambda_{2}(\lambda_{2}-\lambda_{1})}\neq0$ and we obtain a polynomial of fifth grade  on the variable $\nu$.

Note that  $\nu\equiv0$ is a solution for that equation and  $T$ is a  principal direction. If there are other solutions for that equation, also in that case   $\nu$ is  constant  and $T$ is a   principal direction.

In this way we conclude that  in case  $g=2$ and  $n=2$, the field  $T$ is a  principal direction.

Let us analyze now what occurs if $\nu\equiv0$ and  $\nu\neq0$.

If  $\nu\neq0$ the surface is locally given by Theorem \ref{ctdp}. If  $\nu=0$ the  surface is a cylinder  over a curve, $f(M^{2})=\alpha\times\mathbb{R}$, where  $\alpha$ is a circle  non totally geodesic in  $\mathbb{S}^{2}$, if $c=1$, or $\alpha$ is a equidistant curve, a horocycle or a hyperbolic circle  in  $\mathbb{H}^{2}$, if  $c=-1$.

Second case: $n\geq3$ \\
If  $\nu(p) \neq 0$, for all $p\in M$, from Theorem \ref{m1},  the immersion $f$ has at least one   principal curvature with  multiplicity one. As  $n\geq 3$ and  $g=2$ there exist only one principal curvature with multiplicity  $1$  and from  Proposition  \ref{dp}, $T$ is a  principal direction corresponding to that curvature. Then  $f$ is given by Theorem  \ref{ctdp}. From  Remark \ref{obsrot} it is a rotational hypersurface.

If $\nu\equiv0$ then  $f(M^{n})$ is an open subset  of the  Riemannian product $M^{n-1}\times\mathbb{R}$, where  $M^{n-1}$ is a hypersurface of  $\mathbb{Q}^{n}_{c}$. Since the principal curvature corresponding to the factor  $\mathbb{R}$ is null, using the classification of the isoparametric  hypersurfaces in  $\mathbb{Q}^{n}_{c}$, see \cite[Theorem 5]{DF} and \cite[p.4]{TC} the other  curvature may not be zero. So  $M^{n-1}$ must be an isoparametric umbilical and non totally geodesic  hypersurface in  $\mathbb{Q}^{n}_{c}$. This proves  item (ii).

(iii) Suppose  now   $g=3$. According to the considered dimension, there are three possibilities  for the multiplicities  of principal curvatures: two  curvatures with  multipliciity $1$,  just one curvature with multiplicity one and all the curvatures  with multiplicities $\geq2$. Let us analyze each case.

First case:  Suppose that two of the curvatures have multiplicity one.

From  Proposition  \ref{dcpm1}, we don't have  $\nu(p)\neq0$, for all $p\in M^{n}$. Therefore  $\nu\equiv0$ and consequently $f(M^{n})$ is an open subset of the  Riemannian product  $M^{n-1}\times\mathbb{R}$. Then one of the curvatures must be zero.

If  $c=-1$, $\lambda=0$ and the principal curvatures  of an isoparametric hypersurface  with $g=2$ in $\mathbb{H}^{n}$ satisfy $\mu\gamma=1$ and  $\lambda=0$ has multiplicity one. Then  $f(M^{n})$ is locally the product  $\mathbb{S}^{1}\times\mathbb{H}^{n-2}\times\mathbb{R}$ or $\mathbb{S}^{n-2}\times\mathbb{H}^{1}\times\mathbb{R}$.

If $c=1$ and  $n\geq4$, $f(M^{n})$ is an open subset  $\mathbb{S}^{1}\times\mathbb{S}^{n-2}\times\mathbb{R}$. Moreover for  $n=4$ it may also occur that  $f(M^{n})$ is  locally given by $M^{3}\times\mathbb{R}$, where  $M^{3}$ is a tube over a Veronese surface in $\mathbb{S}^{4}$.

Second case:  Suppose that  just one curvature has  multiplicity  $1$.

If  $\nu(p)\neq0$,  for all $p\in M^{n}$, from Proposition  \ref{dp} the field  $T$ is a principal  direction and from  Theorem \ref{ctdp} we get  $c=-1$ and  $g=2$, which may not occur. Then  $\nu\equiv0$ and  $f(M^{n})$ is  an open subset  of the  Riemannian product $M^{n-1}\times\mathbb{R}$. In this case  the curvature with  multiplicity  $1$  is $\lambda=0$.

Now let us  explicit  $M^{n-1}$ if  $c=-1$ and  $c=1$.

If  $c=-1$,  $f(M^{n})$ is an open subset of  $\mathbb{S}^{k}\times\mathbb{H}^{n-k-1}\times\mathbb{R}$, with  $k\geq2$ and  $n\geq5$.
If  $c=1$,  $f(M^{n})$ is an open subset of   $\mathbb{S}^{k}\times\mathbb{S}^{n-k-1}\times\mathbb{R}$, with  $k\geq2$ and  $n\geq5$.

Third  case:  Suppose that all the curvatures have multiplicity   $\geq2$.

From Remark \ref{obs31}, we get  $\nu\equiv0$ and one the principal curvatures is  $\lambda=0$. Moreover, $f(M^{n})$ is an open subset  of the  Riemannian product $M^{n-1}\times\mathbb{R}$.

Then if  $c=-1$, by \cite[Theorem 5]{DF}, there does not exist  isoparametric hypersurfaces, non totally geodesic in  $\mathbb{H}^{n}$ with a principal curvature equal to zero. Otherwise, $\lambda=0$ would have multiplicity one,  which may not occur.

If  $c=1$, the isoparametric hypersurfaces in $\mathbb{S}^{n}$ with  $g=3$, are the Cartan's hypersurfaces.  They have one principal curvature equal to zero.  Then that case occur  for  $n\in\{7,13,25\}$ and  $f(M^{n})$ is locally given by the product $M^{n-1}\times\mathbb{R}$, where  $M^{n-1}$ is a Cartan's hypersurface in  $\mathbb{S}^{n}$.
\end{proof}

\begin{remark} The hypersurfaces classified in  Theorem \ref{teoclass} have function  $\nu$ constant and from   Corollary \ref{haci} they are isoparametric in  $\QR$.
\end{remark}

\section*{Acknowledgements}
The second author was partially supported by CAPES and CNPq, Brazil.

\bibliographystyle{amsplain}

\end{document}